\documentclass[a4paper,11pt,oneside,reqno]{amsproc}
\allowdisplaybreaks[3] 

\usepackage{amsfonts}
\usepackage{amsmath}   
\usepackage{amssymb}
\usepackage{amsthm}

\usepackage{tikz}

\usepackage{lmodern}

\usepackage{url}

\newcommand{\dint}{\displaystyle\int}

\theoremstyle{plain}
\newtheorem{theorem}{Theorem}

\newtheorem{lemma}[theorem]{Lemma}
\newtheorem{corollary}[theorem]{Corollary}

\theoremstyle{definition}

\newcommand{\RR}{{\mathbb R}}

\newcommand{\NN}{{\mathbb N}}

\newcommand{\cD}{{\mathcal D}}

\renewcommand{\Tilde}{\widetilde}

\DeclareMathOperator{\spec}{spec}
\DeclareMathOperator{\supp}{supp}

\newcommand{\dd}{\,\mathrm{d}}

\begin{document}
	\title[]{Peculiar behavior of the principal Laplacian eigenvalue for large negative Robin parameters}

	\author[C. Dietze]{Charlotte Dietze}
	\address{Ludwig-Maximilians-Universit\"at M\"unchen\\Fakult\"at f\"ur Mathematik, Informatik und Statistik\\ Mathematisches Institut\\ Theresienstr.~39\\ 80333 M\"unchen\\ Germany \&
	 Institut des Hautes \'Etudes Scientifiques, 35 Route de Chartres, 91440 Bures-sur-Yvette, France}
	\email{dietze@mathematik.uni-muenchen.de}
	\urladdr{https://www.mathematik.uni-muenchen.de/~dietze/}
	
	\author[K. Pankrashkin]{Konstantin Pankrashkin}
	\address{Carl von Ossietzky Universit\"at Oldenburg\\Fakult\"at V -- Mathematik und Naturwissenschaften\\ Institut f\"ur Mathematik\\ Ammer\-l\"an\-der Heerstr.~114--118\\ 26129 Oldenburg\\ Germany}	\email{konstantin.pankrashkin@uol.de}
	\urladdr{http://uol.de/pankrashkin/}

	%
	%
	%
	%
	%
	
	\begin{abstract}
		Let $\Omega\subset\mathbb{R}^n$ with $n\ge 2$ be a bounded Lipschitz domain with  outer unit normal $\nu$.
		For $\alpha\in\mathbb{R}$ let $R_\Omega^\alpha$ be the Laplacian in $\Omega$ with the Robin boundary condition $\partial_\nu u+\alpha u=0$, and denote by $E(R^\alpha_\Omega)$ its principal eigenvalue.	In 2017 Bucur, Freitas and Kennedy stated the following open question:
		\emph{Does the limit of the ratio $E(R_\Omega^\alpha)/ \alpha^2$ for $\alpha\to-\infty$ always
			exist?} We give a negative answer.
	\end{abstract}
	
	\maketitle

	\section{Introduction}
	
	Let $n\ge 2$ and $\Omega\subset\RR^n$ be a domain with a reasonably regular boundary (for example, Lipschitz, bounded or unbounded with a suitable behavior at infinity) and outer unit normal $\nu$. For $\alpha\in\RR$ denote by $R^\alpha_\Omega$ the Laplacian in $\Omega$, $u\mapsto-\Delta u$, with the Robin boundary condition
	$\partial_\nu u+\alpha u=0$ on $\partial \Omega$, which is rigorously defined
	as the unique self-adjoint operator in $L^2(\Omega)$ generated by the hermitian sesquilinear form $r^\alpha_\Omega$
	given by
	\[
	r^\alpha_\Omega(u,u):=\dint_\Omega |\nabla u|^2\dd x+\alpha\dint_{\partial\Omega}|u|^2\dd \sigma, \quad \cD(r^\alpha_\Omega):=H^1(\Omega),
	\]
	where by $\dd x$ and $\dd\sigma$ we denote the Lebesgue measure and the hypersurface measure in $\RR^n$ respectively, the symbol $H^1$ stands for the first-order Sobolev space, and the restriction of $u\in H^1(\Omega)$ on $\partial \Omega$ in the second integral on the right-hand side is understood in the sense of the standard trace theorems for Sobolev spaces. Further denote
	\[
	E(R^\alpha_\Omega):=\inf\spec R^\alpha_\Omega\equiv 
	\inf_{u\in \cD(r^\alpha_\Omega), \, u\ne 0}\dfrac{r^\alpha_\Omega(u,u)}{\|u\|^2_{L^2(\Omega)}},
	\]
	in particular, the quantity $E(R^\alpha_\Omega)$ is simply the lowest eigenvalue of $R^\alpha_\Omega$ for bounded Lipschitz domains $\Omega$.
	
	The dependence of $E(R^\alpha_\Omega)$ on the Robin parameter $\alpha$ and the geometric properties of $\Omega$ has been given a considerable attention during the last decade, see the reviews in \cite{DFK,kobp}.
	Of particular interest is the case of large negative $\alpha$ due to its applications to some reaction-diffusion equations and
	the appearance of boundary effects \cite{los,LP}. If $\Omega$ is a bounded Lipschitz domain, then for some $c_\Omega\ge 1$ one has the two-sided estimate
	\[
	-c_\Omega\alpha^2\le E(R^\alpha_\Omega)\le-\alpha^2 \text{ for }\alpha<0 \text{ with } |\alpha| \text{ large enough, }
	\]
	see~\cite[Prop.~4.12]{DFK} for the upper bound and~\cite[Lem.~2.7]{kobp} for the lower bound. Remark that these asymptotic estimates fail for non-Lipschitz domains, see \cite{vogel} and references therein.
	On the other hand, as shown in \cite{bp,LP}, if $\Omega$ has a better regularity (for example, if it is a so-called corner domain), then a more precise asymptotic expansion holds,
	\begin{equation}
		\label{asymp00}
		E(R^\alpha_\Omega)=-C_\Omega \alpha^2+o(\alpha^2) \text{ for }\alpha\to-\infty,
	\end{equation}
	with some constant $C_\Omega\ge 1$, and $C_\Omega=1$ if $\Omega$ is a bounded $C^1$ domain \cite{lz}.
	We refer to \cite{DFK,kobp} for a review of further asymptotic results for $E(R_\Omega^\alpha)$ under various additional assumptions on the regularity and geometric properties of $\Omega$, and we also mention the challenging case of large complex $\alpha$ \cite{bkl,krrs}.
	
	Based on the above observations, Bucur, Freitas and  Kennedy asked in \cite[Open question 4.17]{DFK} whether the asymptotic expansion \eqref{asymp00}  holds for any bounded Lipschitz domain $\Omega$ with a suitable constant $C_\Omega$. In this paper we give a \emph{negative} answer. More specifically, our main result reads as follows:
	\begin{theorem}[Existence of ``bad'' domains]\label{thm1}
		There are
		\begin{itemize}
			\item a bounded Lipschitz domain $\Omega\subset\RR^2$,
			\item two disjoint compact intervals $I', I''\subset (-\infty,0)$,
			\item two sequences $(\beta'_k)_{k\in\NN}, (\beta''_k)_{k\in\NN}\subset(-\infty,0)$ converging to $-\infty$,
		\end{itemize}			
		such that 
		\[
		\dfrac{E\big(R_\Omega^{\beta'_k}\big)}{(\beta'_k)^2}\in I' \quad \text{and}\quad
		\dfrac{E\big(R_\Omega^{\beta''_k}\big)}{(\beta''_k)^2}\in I'' \quad \text{for any } k\in\NN.
		\]
	\end{theorem}
	For	a domain $\Omega$ as in Theorem~\ref{thm1} the limit $\lim\limits_{\alpha\to-\infty}E(R_\Omega^\alpha)/\alpha^2$ does not exist and the asymptotics \eqref{asymp00} fails. By taking the direct products of $\Omega$ with finite intervals one produces counterexamples to \eqref{asymp00} in any dimension $n\ge 2$.

	The rest of the paper is devoted to the proof of Theorem~\ref{thm1}. In Section~\ref{sec2} we construct
	a ``building block'' which is an unbounded domain such that the lowest eigenvalue of a Laplacian with mixed Robin-Dirichlet/Neumann boundary conditions
	is nicely controlled in terms of some geometric parameters. In Section \ref{ssec31} we construct first an unbounded ``bad'' domain with Lipschitz boundary by gluing together
	an infinite number of suitably scaled copies of the building blocks. Using adapted truncations of the unbounded ``bad'' domain we construct a bounded ``bad'' domain $\Omega$
	satisfying the conditions of Theorem~\ref{thm1} in Subsection~\ref{ssec32} (in fact, its boundary is smooth except at a single point). In Subsection \ref{ssec33} we show that the constructions can be adapted
	to produce another example of a ``bad'' domain which is non-compact but has a $C^\infty$ boundary.

	\section{Preparatory constructions}\label{sec2}
	
	Let us recall the following known facts, see \cite[Ex.~2.4 and Ex.~2.5]{LP}.

	\begin{lemma}\label{lem-halfline}
		For any $f\in H^1(0,\infty)$ and $\alpha<0$, we have
		\[
		\int_0^\infty \big|f'(t)\big|^2\dd t+\alpha\big|f(0)\big|^2\ge -\alpha^2 \int_0^\infty \big|f(t)\big|^2\dd t,
		\]
		and the equality is attained for the function $f:\,t\mapsto \exp (\alpha t)$.
	\end{lemma}

	\begin{figure}[t]
		\centering

		\tikzset{every picture/.style={line width=0.75pt}} 
		
		\scalebox{0.65}{\begin{tikzpicture}[x=0.75pt,y=0.75pt,yscale=-1,xscale=1]
				
				\draw  [color={rgb, 255:red, 155; green, 155; blue, 155 }  ,draw opacity=1 ][fill={rgb, 255:red, 155; green, 155; blue, 155 }  ,fill opacity=1 ] (189.86,75.69) -- (239.75,173.5) -- (144,173.3) -- cycle ;
				\draw [color={rgb, 255:red, 74; green, 74; blue, 74 }  ,draw opacity=1 ] (89,75.3) -- (290.6,75.3)(190.4,41.5) -- (190.4,174.09) (283.6,70.3) -- (290.6,75.3) -- (283.6,80.3) (185.4,48.5) -- (190.4,41.5) -- (195.4,48.5)  ;
				\draw [line width=3]    (141.25,173.5) -- (190.25,75) -- (239.75,173.5) ;
				\draw  [draw opacity=0] (210.61,120.41) .. controls (204.39,123.2) and (197.5,124.75) .. (190.25,124.75) .. controls (183,124.75) and (176.11,123.2) .. (169.89,120.41) -- (190.25,75) -- cycle ; \draw  [color={rgb, 255:red, 0; green, 0; blue, 0 }  ,draw opacity=1 ] (210.61,120.41) .. controls (204.39,123.2) and (197.5,124.75) .. (190.25,124.75) .. controls (183,124.75) and (176.11,123.2) .. (169.89,120.41) ;

				\draw (267,54.8) node [anchor=north west][inner sep=0.75pt]  [color={rgb, 255:red, 128; green, 128; blue, 128 }  ,opacity=1 ]  {$x$};
				\draw (201,43.4) node [anchor=north west][inner sep=0.75pt]  [color={rgb, 255:red, 128; green, 128; blue, 128 }  ,opacity=1 ]  {$y$};
				\draw (171,124.4) node [anchor=north west][inner sep=0.75pt]    {$\theta $};
				\draw (196,124.4) node [anchor=north west][inner sep=0.75pt]    {$\theta $};
				\draw (155.6,148) node [anchor=north west][inner sep=0.75pt]  [font=\large]  {$S_{\theta }$};

		\end{tikzpicture}}\caption{\small The infinite sector $S_\theta$ in Lemma \ref{lem-sector}.}\label{fig-sector}
		
	\end{figure}
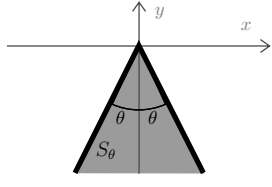

	\begin{lemma}\label{lem-sector}
		For any $\theta\in (0,\frac{\pi}{2})$ denote
		\[
		S_\theta:=\big\{(x,y)\in\RR^2:\  y< -|x|\cot \theta\big\},
		\]
		which  is an infinite sector with opening angle $2\theta$, see Fig.~\ref{fig-sector}.
		Then for any $v\in H^1(S_\theta)$ and any $\alpha<0$,
		\[
		\iint_{S_\theta}|\nabla v|^2\dd x\dd y+\alpha \int_{\partial S_\theta} |v|^2\dd\sigma\ge -\dfrac{\alpha^2}{\sin^2\theta} \iint_{S_\theta}|v|^2\dd x\dd y,
		\]
		and the equality is attained for the function
		\[
		v:\ (x,y)\mapsto \exp\Big(-\dfrac{\alpha y}{\sin\theta}\Big).
		\]
	\end{lemma}

	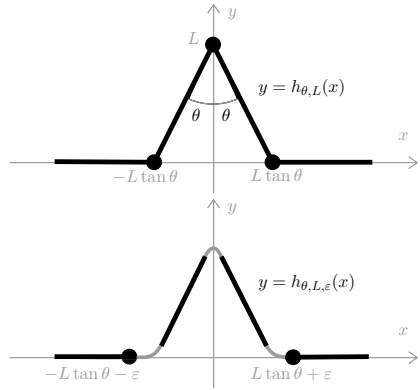
\begin{figure}[b]
		\centering

		\tikzset{every picture/.style={line width=0.75pt}} 

		\scalebox{0.6}{\begin{tikzpicture}[x=0.75pt,y=0.75pt,yscale=-1,xscale=1]
				
				\draw [color={rgb, 255:red, 155; green, 155; blue, 155 }  ,draw opacity=1 ] (1,140.84) -- (340,140.84)(170.75,8.5) -- (170.75,166.5) (333,135.84) -- (340,140.84) -- (333,145.84) (165.75,15.5) -- (170.75,8.5) -- (175.75,15.5)  ;
				\draw [line width=3]    (38.5,140) -- (121.25,140.5) -- (170.25,42) -- (219.75,140.5) -- (302.75,140.5) ;
				\draw  [draw opacity=0] (190.61,87.41) .. controls (184.39,90.2) and (177.5,91.75) .. (170.25,91.75) .. controls (163,91.75) and (156.11,90.2) .. (149.89,87.41) -- (170.25,42) -- cycle ; \draw  [color={rgb, 255:red, 155; green, 155; blue, 155 }  ,draw opacity=1 ] (190.61,87.41) .. controls (184.39,90.2) and (177.5,91.75) .. (170.25,91.75) .. controls (163,91.75) and (156.11,90.2) .. (149.89,87.41) ;  
				\draw  [fill={rgb, 255:red, 0; green, 0; blue, 0 }  ,fill opacity=1 ] (164.13,41.88) .. controls (164.13,38.49) and (166.87,35.75) .. (170.25,35.75) .. controls (173.63,35.75) and (176.38,38.49) .. (176.38,41.88) .. controls (176.38,45.26) and (173.63,48) .. (170.25,48) .. controls (166.87,48) and (164.13,45.26) .. (164.13,41.88) -- cycle ;
				\draw  [fill={rgb, 255:red, 0; green, 0; blue, 0 }  ,fill opacity=1 ] (213.63,140.5) .. controls (213.63,137.12) and (216.37,134.38) .. (219.75,134.38) .. controls (223.13,134.38) and (225.88,137.12) .. (225.88,140.5) .. controls (225.88,143.88) and (223.13,146.63) .. (219.75,146.63) .. controls (216.37,146.63) and (213.63,143.88) .. (213.63,140.5) -- cycle ;
				\draw  [fill={rgb, 255:red, 0; green, 0; blue, 0 }  ,fill opacity=1 ] (94.5,303.5) .. controls (94.5,300.12) and (97.24,297.38) .. (100.63,297.38) .. controls (104.01,297.38) and (106.75,300.12) .. (106.75,303.5) .. controls (106.75,306.88) and (104.01,309.63) .. (100.63,309.63) .. controls (97.24,309.63) and (94.5,306.88) .. (94.5,303.5) -- cycle ;
				\draw [color={rgb, 255:red, 155; green, 155; blue, 155 }  ,draw opacity=1 ] (1,304.34) -- (340,304.34)(170.75,172) -- (170.75,330) (333,299.34) -- (340,304.34) -- (333,309.34) (165.75,179) -- (170.75,172) -- (175.75,179)  ;
				\draw [color={rgb, 255:red, 155; green, 155; blue, 155 }  ,draw opacity=1 ][line width=2.25]    (299.75,303.5) .. controls (272.25,304) and (235.25,304.5) .. (230.75,304) .. controls (226.25,303.5) and (219.75,303.5) .. (214.75,294.5) .. controls (209.75,285.5) and (180.75,227.5) .. (176.25,218) .. controls (171.75,208.5) and (166.75,213) .. (163.75,219.5) .. controls (160.75,226) and (131.25,286.5) .. (128.25,292) .. controls (125.25,297.5) and (121.25,304) .. (112.75,304) .. controls (104.25,304) and (82.25,305) .. (38.5,303.5) ;
				\draw [line width=3]    (38.5,303.5) -- (106.75,303.5) ;
				\draw [line width=3]    (230.75,304) -- (299.75,303.5) ;
				\draw [line width=3]    (127.25,294.5) -- (163.75,219.5) ;
				\draw [line width=3]    (214.75,294.5) -- (177.17,219.67) ;
				\draw  [fill={rgb, 255:red, 0; green, 0; blue, 0 }  ,fill opacity=1 ] (230.75,304) .. controls (230.75,300.62) and (233.49,297.88) .. (236.88,297.88) .. controls (240.26,297.88) and (243,300.62) .. (243,304) .. controls (243,307.38) and (240.26,310.13) .. (236.88,310.13) .. controls (233.49,310.13) and (230.75,307.38) .. (230.75,304) -- cycle ;
				\draw  [fill={rgb, 255:red, 0; green, 0; blue, 0 }  ,fill opacity=1 ] (115.13,140.5) .. controls (115.13,137.12) and (117.87,134.38) .. (121.25,134.38) .. controls (124.63,134.38) and (127.38,137.12) .. (127.38,140.5) .. controls (127.38,143.88) and (124.63,146.63) .. (121.25,146.63) .. controls (117.87,146.63) and (115.13,143.88) .. (115.13,140.5) -- cycle ;
				
				\draw (323,114.4) node [anchor=north west][inner sep=0.75pt]  [color={rgb, 255:red, 155; green, 155; blue, 155 }  ,opacity=1 ]  {$x$};
				\draw (181,10.4) node [anchor=north west][inner sep=0.75pt]  [color={rgb, 255:red, 155; green, 155; blue, 155 }  ,opacity=1 ]  {$y$};
				\draw (206.5,68.9) node [anchor=north west][inner sep=0.75pt]    {$y=h_{\theta,L}( x)$};
				\draw (151,94.4) node [anchor=north west][inner sep=0.75pt]    {$\theta $};
				\draw (176,94.4) node [anchor=north west][inner sep=0.75pt]    {$\theta $};
				\draw (147.5,29.9) node [anchor=north west][inner sep=0.75pt]  [color={rgb, 255:red, 155; green, 155; blue, 155 }  ,opacity=1 ]  {$L$};
				\draw (200,144.9) node [anchor=north west][inner sep=0.75pt]  [color={rgb, 255:red, 155; green, 155; blue, 155 }  ,opacity=1 ]  {$L\tan \theta $};
				\draw (85,145.4) node [anchor=north west][inner sep=0.75pt]  [color={rgb, 255:red, 155; green, 155; blue, 155 }  ,opacity=1 ]  {$-L\tan \theta $};
				\draw (323,277.9) node [anchor=north west][inner sep=0.75pt]  [color={rgb, 255:red, 155; green, 155; blue, 155 }  ,opacity=1 ]  {$x$};
				\draw (181,173.9) node [anchor=north west][inner sep=0.75pt]  [color={rgb, 255:red, 155; green, 155; blue, 155 }  ,opacity=1 ]  {$y$};
				\draw (206.5,232.4) node [anchor=north west][inner sep=0.75pt]    {$y=h_{\theta,L,\varepsilon }( x)$};
				\draw (28.67,310.73) node [anchor=north west][inner sep=0.75pt]  [color={rgb, 255:red, 155; green, 155; blue, 155 }  ,opacity=1 ]  {$-L\tan \theta -\varepsilon $};
				\draw (200,310.4) node [anchor=north west][inner sep=0.75pt]  [color={rgb, 255:red, 155; green, 155; blue, 155 }  ,opacity=1 ]  {$L\tan \theta +\varepsilon $};
		\end{tikzpicture}}
		
		\caption{The graphs of the functions $h_{\theta,L}$ (above) and $h_{\theta,L,\varepsilon}$ (below). Note that $h_{\theta,L,\varepsilon}$ coincides with $h_{\theta,L}$ except in the $\varepsilon$-neighborhoods of $0$ and $\pm L\tan\theta$: The dark parts of the graph of $h_{\theta,L,\varepsilon}$ coincide with the respective parts of the graph of $h_{\theta,L}$.}
		\label{fig:hh}
	\end{figure}
	
	For $\theta\in\big(0,\frac{\pi}{2}\big)$ and $L>0$ define a function
	\[
	h_{\theta,L}:\RR\to\RR,
	\qquad
	h_{\theta,L}(x):=\begin{cases}
		L-|x|\cot\theta, & |x|\le L\tan \theta,\\
		0, & |x|>L\tan\theta.
	\end{cases}
	\]
	The graph of $h_{\theta,L}$ is shown in Figure \ref{fig:hh}. Pick any $\varphi\in C^\infty_c(\RR)$ with
	\[
	\varphi\ge 0,\quad \supp\varphi\subset(-1,1),\quad \int_\RR \varphi(x)\dd x=1, \quad	\text{$\varphi$ is even,}
	\]
	and for $\varepsilon>0$ consider the functions
	\[
	\varphi_\varepsilon:=\dfrac{1}{\varepsilon}\varphi\Big(\frac{\cdot}{\varepsilon}\Big),
	\qquad
	h_{\theta,L,\varepsilon}:=h_{\theta,L}\star \varphi_\varepsilon,
	\]
	with $\star$ being the convolution product. Then $h_{\theta,L,\varepsilon}$ is 
	$C^\infty$ with
	\begin{equation}
		\label{eq-hbound}
		\|h'_{\theta,L,\varepsilon}\|_\infty=\|h'_{\theta,L}\star \varphi_{\varepsilon}\|_\infty\le 
		\|h'_{\theta,L}\|_\infty\|\varphi_{\varepsilon}\|_1=\|h'_{\theta,L}\|_\infty=\cot\theta,
	\end{equation}
	and $\|h_{\theta,L,\varepsilon}-h_{\theta,L}\|_\infty\xrightarrow{\varepsilon\to 0^+}0$.
	Moreover, as $h_{\theta,L}$ is a piecewise affine function,
	the function $h_{\theta,L,\varepsilon}$ coincides with $h_{\theta,L}$ 
	except in the $\varepsilon$-neighborhoods of the points  $0$ and $\pm L\cot\theta$ (at which $h_{\theta,L}$ is non-smooth), see Fig.~\ref{fig:hh}.

	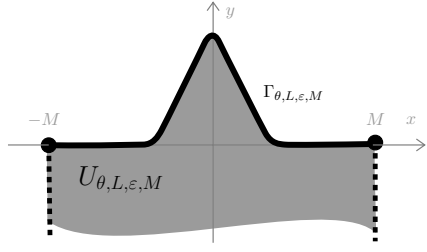
\begin{figure}[t]
		
		\centering

		\tikzset{every picture/.style={line width=0.75pt}} 
		
		\scalebox{0.6}{\begin{tikzpicture}[x=0.75pt,y=0.75pt,yscale=-1,xscale=1]
				
				\draw  [color={rgb, 255:red, 155; green, 155; blue, 155 }  ,draw opacity=1 ][fill={rgb, 255:red, 155; green, 155; blue, 155 }  ,fill opacity=1 ][line width=2.25]  (350,173) .. controls (322.5,173.5) and (281.25,174) .. (276.75,173.5) .. controls (272.25,173) and (265.75,173) .. (260.75,164) .. controls (255.75,155) and (226.75,97) .. (222.25,87.5) .. controls (217.75,78) and (212.75,82.5) .. (209.75,89) .. controls (206.75,95.5) and (177.25,156) .. (174.25,161.5) .. controls (171.25,167) and (167.25,173.5) .. (158.75,173.5) .. controls (150.25,173.5) and (123.5,174) .. (80.5,174) .. controls (80.4,227.4) and (81.2,201.4) .. (80.8,236.6) .. controls (132.4,272.2) and (310,213) .. (351.2,244.6) .. controls (350.4,197) and (350,224.6) .. (350,173) -- cycle ;
				\draw [line width=3]  [dash pattern={on 3.38pt off 3.27pt}]  (80.5,174) -- (81.6,253.4) ;
				\draw [line width=3]  [dash pattern={on 3.38pt off 3.27pt}]  (352,171.8) -- (352,254.6) ;
				\draw  [fill={rgb, 255:red, 0; green, 0; blue, 0 }  ,fill opacity=1 ] (74.38,174) .. controls (74.38,170.62) and (77.12,167.88) .. (80.5,167.88) .. controls (83.88,167.88) and (86.63,170.62) .. (86.63,174) .. controls (86.63,177.38) and (83.88,180.13) .. (80.5,180.13) .. controls (77.12,180.13) and (74.38,177.38) .. (74.38,174) -- cycle ;
				\draw  [fill={rgb, 255:red, 0; green, 0; blue, 0 }  ,fill opacity=1 ] (345.88,171.8) .. controls (345.88,168.42) and (348.62,165.68) .. (352,165.68) .. controls (355.38,165.68) and (358.13,168.42) .. (358.13,171.8) .. controls (358.13,175.18) and (355.38,177.93) .. (352,177.93) .. controls (348.62,177.93) and (345.88,175.18) .. (345.88,171.8) -- cycle ;
				\draw [color={rgb, 255:red, 128; green, 128; blue, 128 }  ,draw opacity=1 ] (46.6,173.8) -- (393.2,173.8)(217.2,54.2) -- (217.2,255.8) (386.2,168.8) -- (393.2,173.8) -- (386.2,178.8) (212.2,61.2) -- (217.2,54.2) -- (222.2,61.2)  ;
				\draw [color={rgb, 255:red, 0; green, 0; blue, 0 }  ,draw opacity=1 ][line width=3.75]    (350,173) .. controls (322.5,173.5) and (281.25,174) .. (276.75,173.5) .. controls (272.25,173) and (265.75,173) .. (260.75,164) .. controls (255.75,155) and (226.75,97) .. (222.25,87.5) .. controls (217.75,78) and (212.75,82.5) .. (209.75,89) .. controls (206.75,95.5) and (177.25,156) .. (174.25,161.5) .. controls (171.25,167) and (167.25,173.5) .. (158.75,173.5) .. controls (150.25,173.5) and (125.6,174.6) .. (80.5,174) ;
				
				\draw (376.2,148.2) node [anchor=north west][inner sep=0.75pt]  [color={rgb, 255:red, 155; green, 155; blue, 155 }  ,opacity=1 ]  {$x$};
				\draw (226.2,57.4) node [anchor=north west][inner sep=0.75pt]  [color={rgb, 255:red, 155; green, 155; blue, 155 }  ,opacity=1 ]  {$y$};
				\draw (257.7,122.3) node [anchor=north west][inner sep=0.75pt]    {$\Gamma _{\theta ,L,\varepsilon ,M}$};
				\draw (103.4,190) node [anchor=north west][inner sep=0.75pt]  [font=\LARGE]  {$U_{\theta ,L,\varepsilon ,M}$};
				\draw (61.7,145.8) node [anchor=north west][inner sep=0.75pt]  [color={rgb, 255:red, 155; green, 155; blue, 155 }  ,opacity=1 ]  {$-M$};
				\draw (342.9,145.4) node [anchor=north west][inner sep=0.75pt]  [color={rgb, 255:red, 155; green, 155; blue, 155 }  ,opacity=1 ]  {$M$};

		\end{tikzpicture}}
		
		\caption{The domain $U_{\theta,L,\varepsilon,M}$. Its boundary is decomposed into the part $\Gamma_{\theta,L,\varepsilon,M}$ (bold curve) and the side boundary (dotted lines).}\label{fig:ug}

	\end{figure}
	
	Introduce a additional parameter $M\ge L\tan\theta$
	and define a domain
	\begin{align*}
		U_{\theta,L,\varepsilon,M}&:=\big\{
		(x,y)\in\RR^2:\  x\in(-M,M),\ y< h_{\theta,L,\varepsilon}(x)\big\}
		\intertext{and a subset of its boundary}
		\Gamma_{\theta,L,\varepsilon,M}&:=\big\{
		(x,y)\in\RR^2:\  x\in(-M,M),\ y=h_{\theta,L,\varepsilon}(x)
		\big\}\subset \partial U_{\theta,L,\varepsilon,M},
	\end{align*}
	see Fig.~\ref{fig:ug}.
	In what follows we will be concerned with the study of Laplacians in scaled copies of $U_{\theta,L,\varepsilon,M}$ with special boundary conditions at different parts of the boundary. Namely,
	for $\alpha\in\RR$ (Robin parameter), $c>0$ (scaling parameter) and $\nu:=$\,the outer unit normal we denote:
	\begin{align*}
		Q^{\alpha,c}_{\theta,L,\varepsilon,M}&:=\parbox[t]{95mm}{the positive Laplacian in $cU_{\theta,L,\varepsilon,M}$ with the Robin boundary condition $\partial_\nu u+\alpha u=0$ on $c\Gamma_{\theta,L,\varepsilon,M}$ and \emph{Neumann} boundary condition on the remaining boundary.}
		\intertext{More rigorously,	$Q^{\alpha,c}_{\theta,L,\varepsilon,M}$ is the self-adjoint  operator in $L^2(cU_{\theta,L,\varepsilon,M})$ defined by the hermitian sesquilinear form $q^{\alpha,c}_{\theta,L,\varepsilon,M}$ given by}
		q^{\alpha,c}_{\theta,L,\varepsilon,M}(u,u)&=\iint_{cU_{\theta,L,\varepsilon,M}}|\nabla u|^2\dd x\dd y+\alpha\int_{c\Gamma_{\theta,L,\varepsilon,M}}|u|^2\dd \sigma
		\intertext{	on the domain $\cD(q^{\alpha,c}_{\theta,L,\varepsilon,M})=H^1(cU_{\theta,L,\varepsilon,M})$,}
		\Tilde Q^{\alpha,c}_{\theta,L,\varepsilon,M}&:= \parbox[t]{95mm}{the positive Laplacian in $cU_{\theta,L,\varepsilon,M}$ with the Robin boundary condition $\partial_\nu u+\alpha u=0$ on $c\Gamma_{\theta,L,\varepsilon,M}$ and \emph{Dirichlet} boundary condition on the remaining boundary.}
		\intertext{In other words, $\Tilde Q^{\alpha,c}_{\theta,L,\varepsilon,M}$ is the self-adjoint  operator in $L^2(cU_{\theta,L,\varepsilon,M})$ defined by the sesquilinear form $\Tilde q^{\,\alpha,c}_{\theta,L,\varepsilon,M}$ which is the restriction of the above form $q^{\alpha,c}_{\theta,L,\varepsilon,M}$
			on the domain}
		\cD(\Tilde q^{\alpha,c}_{\theta,L,\varepsilon,M})&:=\big\{
		u\in H^1(cU_{\theta,L,\varepsilon,M}):\, u=0 \text{ on } \partial (cU_{\theta,L,\varepsilon,M})\setminus 
		c\Gamma_{\theta,L,\varepsilon,M}\big\},
	\end{align*}
	where the condition $u=0$ is understood in the sense of Sobolev traces (the same convention will be used in the definitions of further operators and forms below).
	For a lower semibounded self-adjoint operator $A$ we will denote
	\[
	E(A):=\inf\spec A.
	\]

	Due  to the min-max principle,
	\begin{equation}
		\label{eq-qq}
		E(Q^{\alpha,c}_{\theta,L,\varepsilon,M})\le E(\Tilde Q^{\alpha,c}_{\theta,L,\varepsilon,M}).
	\end{equation}
	Furthermore,  a simple scaling argument shows that the operator $Q^{t\alpha,c}_{\theta,L,\varepsilon,M}$ is unitarily equivalent to $t^2 Q^{\alpha,tc}_{\theta,L,\varepsilon,M}$ and, similarly, the operator $\Tilde Q^{t\alpha,c}_{\theta,L,\varepsilon,M}$ is unitarily equivalent to $t^2 \Tilde Q^{\alpha,tc}_{\theta,L,\varepsilon,M}$, so we have
	\begin{equation}
	E(Q^{t\alpha,c}_{\theta,L,\varepsilon,M})=t^2 E(Q^{\alpha,tc}_{\theta,L,\varepsilon,M}),\quad
	E(\Tilde Q^{t\alpha,c}_{\theta,L,\varepsilon,M})=t^2 E(\Tilde Q^{\alpha,tc}_{\theta,L,\varepsilon,M}) \label{eq-scale}
	\end{equation}
	for any admissible choice of $(\theta,L,\varepsilon,M,\alpha,c)$ and any $t>0$. 
	
	\begin{lemma}\label{lem4}
		For any $\theta\in(0,\frac{\pi}{2})$ and $\delta>0$ there exist $\varepsilon_0>0$ and $L_0>0$ such that
		for all $\varepsilon\in (0,\varepsilon_0)$, $L\in(L_0,\infty)$ and $M\in (L\tan\theta,\infty)$ it holds that
		\[
			E(Q^{-1,1}_{\theta,L,\varepsilon,M}), E(\Tilde Q^{-1,1}_{\theta,L,\varepsilon,M})\in \Big[ -\dfrac{1}{\sin^2\theta},-\dfrac{1}{\sin^2\theta}+\delta\Big].
		\]
	\end{lemma}	
	
	\begin{proof}
		Let $u\in \cD(q^{-1,1}_{\theta,L,\varepsilon,M})$, then
		\begin{align*}
			q^{-1,1}_{\theta,L,\varepsilon,M}(u,u)&=\iint_{U_{\theta,L,\varepsilon,M}}|\nabla u|^2\dd x\dd y-\int_{\Gamma_{\theta,L,\varepsilon,M}}|u|^2\dd \sigma\\
			&\ge \iint_{U_{\theta,L,\varepsilon,M}}|\partial_y u|^2\dd x\dd y-\int_{\Gamma_{\theta,L,\varepsilon,M}}|u|^2\dd \sigma\\
			&=\int_{-M}^M\bigg[
			\int_{-\infty}^{h_{\theta,L,\varepsilon}(x)} \big|\partial_y u(x,y)\big|^2\dd y\\
			&\quad-\sqrt{1+\big|h'_{\theta,L,\varepsilon}(x)\big|^2} \,\Big|u\big(x,h_{\theta,L,\varepsilon}(x)\big)\Big|^2
			\bigg]\dd x.
		\end{align*}
		By \eqref{eq-hbound}, for any $x\in(-M,M)$ we have
		\[
		\sqrt{1+\big|h'_{\theta,L,\varepsilon}(x)\big|^2}\le \sqrt{1+\cot ^2\theta}=\dfrac{1}{\sin\theta},
		\]
		therefore,
		\begin{align*}
			\int_{-\infty}^{h_{\theta,L,\varepsilon}(x)} &\big|\partial_y u(x,y)\big|^2\dd y
			-\sqrt{1+\big|h'_{\theta,L,\varepsilon}(x)\big|^2} \,\Big|u\big(x,h'_{\theta,L,\varepsilon}(x)\big)\Big|^2\\
			&\ge \int_{-\infty}^{h_{\theta,L,\varepsilon}(x)} \big|\partial_y u(x,y)\big|^2\dd y
			-\dfrac{1}{\sin\theta}\Big|u\big(x,h_{\theta,L,\varepsilon}(x)\big)\Big|^2\\
			\text{(use Lemma \ref{lem-halfline})}\quad	&\ge -\dfrac{1}{\sin^2\theta} \int_{-\infty}^{h_{\theta,L,\varepsilon}(x)} \big|u(x,y)\big|^2\dd y,
		\end{align*}
		and we arrive at
		\[
		q^{-1,1}_{\theta,L,\varepsilon,M}(u,u)\ge -\dfrac{1}{\sin^2\theta} \iint_{U_{\theta,L,\varepsilon,M}} \big|u(x,y)\big|^2\dd x\dd y.
		\]
		The min-max principle yields
		\[
		E(Q^{-1,1}_{\theta,L,\varepsilon,M})\ge -\dfrac{1}{\sin^2\theta},
		\]
		and the required lower bound for $E(\Tilde Q^{-1,1}_{\theta,L,\varepsilon,M})$ follows by \eqref{eq-qq}.

		For the upper bound  we will use the observation
		that the shifted domain $U_{\theta,L,\varepsilon,M}-(0,L)$ ``converges'' to the sector $S_\theta$ from Lemma~\ref{lem-sector} as $(\varepsilon,L)$ approaches $(0,\infty)$.
		Pick $\chi\in C^\infty(\RR)$ such that
		\[
		0\le\chi\le 1,\qquad \chi(t)=1 \text{ for all }t \ge -\frac{1}{2},
		\qquad 
		\chi(t)=0 \text{ for all }t\le -\frac{3}{4}.
		\]
		Define the functions
		\begin{align*}
			v:\,&\ S_\theta\ni (x,y)\mapsto \exp\Big(-\dfrac{\alpha y}{\sin\theta}\Big),\\
			v_L:\,&\  U_{\theta,L,\varepsilon,M}\ni(x,y)\mapsto v(x,y-L)\chi\Big(\dfrac{y-L}{L}\Big).
		\end{align*}
		Remark that the support of $v_L$ is contained in the set
		\[
		U_{\theta,L,\varepsilon,M}\cap \Big\{(x,y):\ y\ge \frac{L}{4}\Big\}
		\equiv \Big\{(x,y):\ \frac{L}{4}\le y < h_{\theta,L,\varepsilon}(x)\Big\},
		\]
		which is independent of $M$ and represents a shifted copy of a finite piece of $S_\theta$ with a rounded corner, and that $\|v_L\|_{L^2(U_{\theta,L,\varepsilon,M})}$ and $\Tilde q^{\,-1,1}_{\theta,L,\varepsilon,M}(v_L,v_L)$
		are independent of $M$ as well. A simple direct computation shows:
		\begin{align*}
			\iint_{U_{\theta,L,\varepsilon,M}} |\nabla v_L|^2\dd x\dd y&\xrightarrow{(\varepsilon,L)\to(0,\infty)}
			\iint_{S_\theta} |\nabla v|^2\dd x\dd y,\\
			\iint_{U_{\theta,L,\varepsilon,M}} |v_L|^2\dd x\dd y&\xrightarrow{(\varepsilon,L)\to(0,\infty)}
			\iint_{S_\theta} |v|^2\dd x\dd y,\\
			\int_{\Gamma_{\theta,L,\varepsilon,M}} |v_L|^2\dd \sigma&\xrightarrow{(\varepsilon,L)\to(0,\infty)}
			\int_{\partial S_\theta} |v|^2\dd \sigma,
		\end{align*}
		which yields
		\[
		\dfrac{\Tilde q^{\,-1,1}_{\theta,L,\varepsilon,M}(v_L,v_L)}{\|v_L\|^2_{L^2(U_{\theta,L,\varepsilon,M})}}
		\ \xrightarrow{(\varepsilon,L)\to(0,\infty)}\ \dfrac{\displaystyle\iint_{S_\theta} |\nabla v|^2\dd x\dd y-\int_{\partial S_\theta} |v|^2\dd \sigma}{\displaystyle\iint_{S_\theta} |v|^2\dd x\dd y}=-\dfrac{1}{\sin^2\theta},
		\]
		where the last equality holds by Lemma~\ref{lem-sector}.
		Therefore, for any $\delta>0$ one finds $\varepsilon_0>0$ and $L_0>0$ such that for all $(\varepsilon,L)\in(0,\varepsilon_0)\times (L_0,\infty)$ and all admissible $M$ we have
		\[
		E(\Tilde Q^{-1,1}_{\theta,L,\varepsilon,M})\le
		\dfrac{\Tilde q^{\,-1,1}_{\theta,L,\varepsilon,M}(v_L,v_L)}{\|v_L\|^2_{L^2(U_{\theta,L,\varepsilon,M})}}\le 
		-\dfrac{1}{\sin^2\theta}+\delta,
		\]
		and the required upper bound for $E(Q^{-1,1}_{\theta,L,\varepsilon,M})$ follows by \eqref{eq-qq}.
	\end{proof}

	\begin{lemma}\label{lem5}
		For any $\theta\in(0,\frac{\pi}{2})$, $L\in(0,\infty)$, $\varepsilon>0$
		and $\delta>0$ one can find $M\in (L\tan\theta,\infty)$ and $\alpha_0<0$
		such that for any $\alpha\in (\alpha_0,0)$, we have
		\[
		E(\Tilde Q^{\alpha,1}_{\theta,L,\varepsilon,M})\ge
		E(Q^{\alpha,1}_{\theta,L,\varepsilon,M})\ge -(1+\delta)\alpha^2.
		\]
	\end{lemma}

	\begin{proof}
		The left-hand inequality holds by \eqref{eq-qq}, so only the right-hand one must be shown. Let us choose some admissible $(\theta,L,\varepsilon,M)$ and any $\delta>0$. Pick some $\rho\in (0,1)$: its precise value will be chosen later. For any $\alpha\le 0$ and any $u\in\cD(q^{\alpha,1}_{\theta,L,\varepsilon,M})$ we estimate
		\begin{align*}
			q^{\alpha,1}_{\theta,L,\varepsilon,M}(u,u)&=\iint_{U_{\theta,L,\varepsilon,M}}|\nabla u|^2\dd x\dd y+\alpha\int_{\Gamma_{\theta,L,\varepsilon,M}}|u|^2\dd \sigma\\
			&=\iint_{U_{\theta,L,\varepsilon,M}}\big(|\partial_x u|^2+\rho |\partial_y u|^2\big)\dd x\dd y\\
			&\quad+(1-\rho)	
			\int_{-M}^M\bigg[
			\int_{-\infty}^{h_{\theta,L,\varepsilon}(x)} \big|\partial_y u(x,y)\big|^2\dd y\\
			&\qquad\qquad
			+\dfrac{\alpha}{1-\rho}\sqrt{1+\big|h'_{\theta,L,\varepsilon}(x)\big|^2} \,\Big|u\big(x,h_{\theta,L,\varepsilon}(x)\big)\Big|^2
			\bigg]\dd x.
		\end{align*}
		We estimate the term in the square brackets using Lemma \ref{lem-halfline}:
		\begin{multline*}
			\int_{-\infty}^{h_{\theta,L,\varepsilon}(x)} \big|\partial_y u(x,y)\big|^2\dd y
			+\dfrac{\alpha}{1-\rho}\sqrt{1+\big|h'_{\theta,L,\varepsilon}(x)\big|^2} \,\Big|u\big(x,h_{\theta,L,\varepsilon}(x)\big)\Big|^2\\
			\ge 
			-\dfrac{\alpha^2}{(1-\rho)^2}w_{\theta,L,\varepsilon}(x) \int_{-\infty}^{h_{\theta,L,\varepsilon}(x)} \big|u(x,y)\big|^2\dd y
		\end{multline*}
		with 
		\begin{equation}\label{eq:wdef}
		w_{\theta,L,\varepsilon}(x):=1+\big|h'_{\theta,L,\varepsilon}(x)\big|^2,
		\end{equation}
		which yields
		\begin{equation}
			\label{eq-est01}
			q^{\alpha,1}_{\theta,L,\varepsilon,M}(u,u)\ge 	\iint_{U_{\theta,L,\varepsilon,M}}\Big(|\partial_x u|^2+\rho |\partial_y u|^2 -\dfrac{\alpha^2}{1-\rho}W_{\theta,L,\varepsilon}|u|^2\Big)\dd x\dd y
		\end{equation}	
		with $W_{\theta,L,\varepsilon}(x,y):=w_{\theta,L,\varepsilon}(x)$.
		
			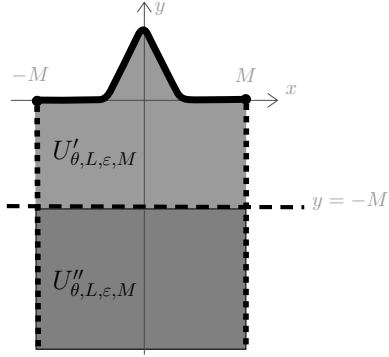
\begin{figure}
			\centering
			\scalebox{0.7}{\begin{tikzpicture}[x=0.75pt,y=0.75pt,yscale=-1,xscale=1]
					
					\draw  [color={rgb, 255:red, 155; green, 155; blue, 155 }  ,draw opacity=1 ][fill={rgb, 255:red, 155; green, 155; blue, 155 }  ,fill opacity=1 ][line width=2.25]  (250.03,77.89) .. controls (234.96,78.16) and (212.36,78.44) .. (209.9,78.16) .. controls (207.43,77.89) and (203.87,77.89) .. (201.13,72.96) .. controls (198.39,68.03) and (182.5,36.25) .. (180.04,31.04) .. controls (177.57,25.84) and (174.83,28.31) .. (173.19,31.87) .. controls (171.55,35.43) and (155.38,68.58) .. (153.74,71.59) .. controls (152.09,74.6) and (149.9,78.16) .. (145.25,78.16) .. controls (140.59,78.16) and (125.93,78.44) .. (102.37,78.44) .. controls (102.32,107.7) and (103.42,226.71) .. (103.2,246) .. controls (131.47,265.51) and (227.43,233.09) .. (250,250.4) .. controls (249.56,224.32) and (250.03,106.16) .. (250.03,77.89) -- cycle ;
					\draw  [draw opacity=0][fill={rgb, 255:red, 128; green, 128; blue, 128 }  ,fill opacity=1 ] (101.6,156) -- (251.2,156) -- (251.2,256.8) -- (101.6,256.8) -- cycle ;
					\draw [color={rgb, 255:red, 74; green, 74; blue, 74 }  ,draw opacity=1 ] (83.8,78) -- (273.7,78)(178.8,7.6) -- (178.8,260.8) (266.7,73) -- (273.7,78) -- (266.7,83) (173.8,14.6) -- (178.8,7.6) -- (183.8,14.6)  ;
					\draw [color={rgb, 255:red, 0; green, 0; blue, 0 }  ,draw opacity=1 ][line width=3.75]    (251.13,77.23) .. controls (236.06,77.51) and (213.46,77.78) .. (210.99,77.51) .. controls (208.53,77.23) and (204.97,77.23) .. (202.23,72.3) .. controls (199.49,67.37) and (183.6,35.59) .. (181.13,30.39) .. controls (178.67,25.18) and (175.93,27.65) .. (174.28,31.21) .. controls (172.64,34.77) and (156.48,67.92) .. (154.83,70.93) .. controls (153.19,73.94) and (151,77.51) .. (146.34,77.51) .. controls (141.68,77.51) and (128.18,78.11) .. (103.47,77.78) ;
					\draw [line width=3]  [dash pattern={on 3.38pt off 3.27pt}]  (102.37,81.79) -- (103.2,258.8) ;
					\draw  [fill={rgb, 255:red, 0; green, 0; blue, 0 }  ,fill opacity=1 ] (99.02,78.44) .. controls (99.02,76.58) and (100.52,75.08) .. (102.37,75.08) .. controls (104.23,75.08) and (105.73,76.58) .. (105.73,78.44) .. controls (105.73,80.29) and (104.23,81.79) .. (102.37,81.79) .. controls (100.52,81.79) and (99.02,80.29) .. (99.02,78.44) -- cycle ;
					\draw [line width=3]  [dash pattern={on 3.38pt off 3.27pt}]  (251.53,80.43) -- (251.2,256.8) ;
					\draw  [fill={rgb, 255:red, 0; green, 0; blue, 0 }  ,fill opacity=1 ] (247.77,77.23) .. controls (247.77,75.38) and (249.27,73.88) .. (251.13,73.88) .. controls (252.98,73.88) and (254.48,75.38) .. (254.48,77.23) .. controls (254.48,79.09) and (252.98,80.59) .. (251.13,80.59) .. controls (249.27,80.59) and (247.77,79.09) .. (247.77,77.23) -- cycle ;
					\draw [line width=2.25]  [dash pattern={on 6.75pt off 4.5pt}]  (80.8,154) -- (292.8,154.8) ;
					
					\draw (278.07,65.67) node [anchor=north west][inner sep=0.75pt]  [color={rgb, 255:red, 155; green, 155; blue, 155 }  ,opacity=1 ]  {$x$};
					\draw (185.09,5.92) node [anchor=north west][inner sep=0.75pt]  [color={rgb, 255:red, 155; green, 155; blue, 155 }  ,opacity=1 ]  {$y$};
					\draw (112,104.2) node [anchor=north west][inner sep=0.75pt]  [font=\Large]  {$U'_{\theta ,L,\varepsilon ,M}$};
					\draw (82.74,52.55) node [anchor=north west][inner sep=0.75pt]  [color={rgb, 255:red, 155; green, 155; blue, 155 }  ,opacity=1 ]  {$-M$};
					\draw (242.07,53.33) node [anchor=north west][inner sep=0.75pt]  [color={rgb, 255:red, 155; green, 155; blue, 155 }  ,opacity=1 ]  {$M$};
					\draw (297.27,142.87) node [anchor=north west][inner sep=0.75pt]  [color={rgb, 255:red, 155; green, 155; blue, 155 }  ,opacity=1 ]  {$y=-M$};
					\draw (112,198.2) node [anchor=north west][inner sep=0.75pt]  [font=\Large]  {$U''_{\theta ,L,\varepsilon ,M}$};

			\end{tikzpicture}}
			\caption{The subdomains $U'_{\theta,L,\varepsilon,M}$ and $U''_{\theta,L,\varepsilon,M}$ of $U_{\theta,L,\varepsilon,M}$.}\label{fig:uprime}
		\end{figure}
		
		Decompose $U_{\theta,L,\varepsilon,M}$ into two parts,
		\begin{align*}
			U'_{\theta,L,\varepsilon,M}&:=U_{\theta,L,\varepsilon,M}\cap\big\{(x,y):\, y>-M\big\},\\
			U''_{\theta,L,\varepsilon,M}&:=U_{\theta,L,\varepsilon,M}\cap\big\{(x,y):\, y<-M\big\},
		\end{align*}
		see Fig.~\ref{fig:uprime},	and consider the hermitian sesquilinear form given by
		\[
		{q'}^{\alpha,1,\rho}_{\theta,L,\varepsilon,M}(u,u):=\iint_{U'_{\theta,L,\varepsilon,M}}\Big(|\partial_x u|^2+\rho |\partial_y u|^2 -\dfrac{\alpha^2}{1-\rho}W_{\theta,L,\varepsilon}|u|^2\Big)\dd x\dd y
		\]
		with $\cD({q'}^{\alpha,1,\rho}_{\theta,L,\varepsilon,M}):=H^1(U'_{\theta,L,\varepsilon,M})$, which induces
		a self-adjoint operator ${Q'}^{\alpha,1,\rho}_{\theta,L,\varepsilon,M}$ in $L^2(U'_{\theta,L,\varepsilon,M})$,
		and the hermitian sesquilinear form defined by
		\[
		{q''}^{\alpha,1,\rho}_{\theta,L,\varepsilon,M}(u,u):=\iint_{U''_{\theta,L,\varepsilon,M}}\Big(|\partial_x u|^2+\rho |\partial_y u|^2 -\dfrac{\alpha^2}{1-\rho}W_{\theta,L,\varepsilon}|u|^2\Big)\dd x\dd y
		\]
		with $\cD({q''}^{\alpha,1,\rho}_{\theta,L,\varepsilon,M}):=H^1(U''_{\theta,L,\varepsilon,M})$, which generates
		a self-adjoint operator ${Q''}^{\alpha,1,\rho}_{\theta,L,\varepsilon,M}$ in $L^2(U''_{\theta,L,\varepsilon,M})$.
		Then Eq. \eqref{eq-est01} reads as
		\begin{align*}
			q^{\alpha,1}_{\theta,L,\varepsilon,M}(u,u)&\ge {q'}^{\alpha,1,\rho}_{\theta,L,\varepsilon,M}(u|_{U'_{\theta,L,\varepsilon,M}},u|_{U'_{\theta,L,\varepsilon,M}})\\
			&\quad +{q''}^{\alpha,1,\rho}_{\theta,L,\varepsilon,M}(u|_{U''_{\theta,L,\varepsilon,M}},u|_{U''_{\theta,L,\varepsilon,M}})
		\end{align*}
		for any $u\in\cD(q^{\alpha,1}_{\theta,L,\varepsilon,M})$, and the min-max principle implies
		\begin{equation}
			\label{eq-infee}
			\begin{aligned}
				E(Q^{\alpha,1}_{\theta,L,\varepsilon,M})&\ge E({Q'}^{\alpha,1,\rho}_{\theta,L,\varepsilon,M}\oplus {Q''}^{\alpha,1,\rho}_{\theta,L,\varepsilon,M})\\
				&\equiv \min\Big\{
				E({Q'}^{\alpha,1,\rho}_{\theta,L,\varepsilon,M}), E({Q''}^{\alpha,1,\rho}_{\theta,L,\varepsilon,M})\Big\}.
			\end{aligned}
		\end{equation}
		
		Remark that $U'_{\theta,L,\varepsilon,M}$ is a bounded Lipschitz domain, so the operator ${Q'}^{\alpha,1,\rho}_{\theta,L,\varepsilon,M}$ has compact resolvent and $E({Q'}^{\alpha,1,\rho}_{\theta,L,\varepsilon,M})$ is simply its lowest eigenvalue.
		Further remark that ${Q'}^{\alpha,1,\rho}_{\theta,L,\varepsilon,M}$ is a type (B) analytic family with respect to $\alpha^2$, see \cite[Chap.~7.4]{kato} and the ``unperturbed''
		operator ${Q'}^{0,1,\rho}_{\theta,L,\varepsilon,M}$ corresponding to $\alpha^2=0$
		is defined by the hermitian sesquilinear form
		\[
		{q'}^{0,1,\rho}_{\theta,L,\varepsilon,M}(u,u):=\iint_{U'_{\theta,L,\varepsilon,M}}\Big(|\partial_x u|^2+\rho |\partial_y u|^2\Big)\dd x\dd y,
		\]
		so its lowest eigenvalue is $0$, and it is simple with
		$v:=|U'_{\theta,L,\varepsilon,M}|^{-\frac{1}{2}}$ being an associated normalized eigenfunction. The first-order perturbation theory, see \cite[Sec.~7.4.6]{kato} shows
		that for $\alpha\to 0^-$ one has
		\begin{gather}
				E({Q'}^{\alpha,1,\rho}_{\theta,L,\varepsilon,M})=-\dfrac{\alpha^2}{1-\rho} B'_M+O(\alpha^4),
			\label{eq-ee1}\\
		B'_M:=\iint_{U'_{\theta,L,\varepsilon,M}} W_{\theta,L,\varepsilon} |v|^2\dd x \dd y\equiv \dfrac{1}{|U'_{\theta,L,\varepsilon,M}|} \iint_{U'_{\theta,L,\varepsilon,M}} W_{\theta,L,\varepsilon}\dd x\dd y.\nonumber
		\end{gather}
		
		Recall that $\|W_{\theta,L,\varepsilon}\|_\infty\le 1+\cot^2\theta$,
		see \eqref{eq-hbound} and \eqref{eq:wdef}, and that $W_{\theta,L,\varepsilon}(x,y)=1$ for $|x|>L\tan\theta+\varepsilon$.
		Due to the inclusion
		\[
		U'_{\theta,L,\varepsilon,M}\cap\big\{(x,y):\,|x|\le L\tan\theta+\varepsilon\big\}\subset(-L\tan\theta-\varepsilon,L\tan\theta+\varepsilon)\times(-M,L),
		\]
		we obtain the upper bound
		\[
		\iint_{U'_{\theta,L,\varepsilon,M}} W_{\theta,L,\varepsilon}\dd x\dd y\le |U'_{\theta,L,\varepsilon,M}| +2(L\tan\theta+\varepsilon)(L+M)\cot^2\theta.
		\]
		Hence, using $|U'_{\theta,L,\varepsilon,M}|\ge 2M^2$ one obtains
		\begin{align*}
			B'_M&\le \dfrac{|U'_{\theta,L,\varepsilon,M}| +2(L\tan\theta+\varepsilon)(L+M)\cot^2\theta}{|U'_{\theta,L,\varepsilon,M}|}\\
			&\le 1+\dfrac{2(L\tan\theta+\varepsilon)(L+M)\cot^2\theta}{|U'_{\theta,L,\varepsilon,M}|}
			= 1+\dfrac{(L\tan\theta+\varepsilon)(L+M)\cot^2\theta}{M^2}.
		\end{align*}
		In particular, we may chose a sufficiently large $M'_\rho>0$ such that
		$B'_M\le 1+\rho$ for any $M>M'_\rho$. Using \eqref{eq-ee1} for any $M>M'_\rho$ we obtain
		\[
		E({Q'}^{\alpha,1,\rho}_{\theta,L,\varepsilon,M})\ge -\dfrac{1+\rho}{1-\rho}\alpha^2+O(\alpha^4) \text{ for }\alpha\to 0^-,
		\]
		and for suitable small $\alpha'(M)<0$ we have then
		\begin{equation}
			\label{eq-ee2}
			E({Q'}^{\alpha,1,\rho}_{\theta,L,\varepsilon,M})\ge -\Big(\dfrac{1+\rho}{1-\rho}+\rho\Big)\alpha^2
			\text{ for all } M>M'_\rho \text{ and } \alpha\in\big(\alpha'(M),0\big).
		\end{equation}

		To analyze $E({Q''}^{\alpha,1,\rho}_{\theta,L,\varepsilon,M})$ we note that 
	$U''_{\theta,L,\varepsilon,M}\equiv (-M,M)\times(-\infty,-M)$
		only depends on $M$, and the expression
		\begin{align*}
			{q''}^{\alpha,1,\rho}_{\theta,L,\varepsilon,M}(u,u)&=\int_{-M}^M\int_{-\infty}^M\Big(|\partial_x u(x,y)|^2+\rho |\partial_y u(x,y)|^2\\
			&\quad -\dfrac{\alpha^2}{1-\rho}w_{\theta,L,\varepsilon}(x)|u(x,y)|^2\Big)\dd y\dd x,
		\end{align*}
		shows that the operator ${Q''}^{\alpha,1,\rho}_{\theta,L,\varepsilon,M}$ admits a separation of variables,
		\[
		{Q''}^{\alpha,1,\rho}_{\theta,L,\varepsilon,M}=A^{\alpha,\rho}_{\theta,L,\varepsilon,M}\otimes 1 +1\otimes \rho N_M,
		\]
		where $A^{\alpha,\rho}_{\theta,L,\varepsilon,\rho}$ is the operator
		\[
		f\mapsto -f''-\dfrac{\alpha^2}{1-\rho} w_{\theta,L,\varepsilon} f
		\]
		in $L^2(-M,M)$ with Neumann boundary conditions and $N_M$ is the operator $g\mapsto-g''$ in $L^2(-\infty,-M)$
		with Neumann boundary condition. In particular,
		\[
		E({Q''}^{\alpha,1,\rho}_{\theta,L,\varepsilon,M})=E(A^{\alpha,\rho}_{\theta,L,\varepsilon,\rho})+E(\rho N_M)
		=E(A^{\alpha,\rho}_{\theta,L,\varepsilon,\rho}).
		\]
		We use again that $A^{\alpha,\rho}_{\theta,L,\varepsilon,\rho}$ is a type (B) analytic family with respect to $\alpha^2$. The ``unperturbed'' operator $A^{0,\rho}_{\theta,L,\varepsilon,\rho}$ corresponding to $\alpha=0$
		has the lowest eigenvalue zero with eigensubspace spanned by the normalized eigenfunction $\Tilde v:=(2M)^{-\frac{1}{2}}$, and the first-order perturbation theory implies
		\begin{align*}
			E(A^{\alpha,\rho}_{\theta,L,\varepsilon,\rho})&=-\dfrac{\alpha^2}{1-\rho}B''_M+O(\alpha^4) \text{ for }\alpha\to 0^-,\\
			B''_M&:=\int_{-M}^M w_{\theta,L,\varepsilon}|\Tilde v|^2\dd x=\dfrac{1}{2M}\int_{-M}^M w_{\theta,L,\varepsilon}(x)\dd x.
		\end{align*}
		Using $\|w_{\theta,L,\varepsilon}\|_\infty\le 1+\cot^2\theta$,
		see \eqref{eq-hbound}, and $w_{\theta,L,\varepsilon}(x)=1$ for $|x|>L\tan\theta+\varepsilon$
		we estimate
		\begin{gather*}
			\int_{-M}^M w_{\theta,L,\varepsilon}(x)\dd x\le 2M+2(L\tan\theta+\varepsilon)\cot^2\theta,\\
			B''_M\le 1+\dfrac{(L\tan\theta+\varepsilon)\cot^2\theta}{M},
		\end{gather*}
		so we can choose a sufficiently large $M''_\rho>0$ such that for any $M>M''_\rho$ we get
		\[
		E(A^{\alpha,\rho}_{\theta,L,\varepsilon,\rho})\ge -\dfrac{1+\rho}{1-\rho}\alpha^2+O(\alpha^4) \text{ for }\alpha\to 0^-.
		\]
		Then for suitable small $\alpha''(M)<0$, we obtain
		\begin{equation}
			\label{eq-ee3}
			\begin{aligned}E({Q''}^{\alpha,1,\rho}_{\theta,L,\varepsilon,M})&=E(A^{\alpha,\rho}_{\theta,L,\varepsilon,\rho})\ge -\Big(\dfrac{1+\rho}{1-\rho}+\rho\Big)\alpha^2\\
				&\text{for all } M>M''_\rho,\ \alpha\in\big(\alpha''(M),0\big).
			\end{aligned}
		\end{equation}
		
		Now we choose $\rho$ sufficiently small to have the inequality
		\[
		\dfrac{1+\rho}{1-\rho}+\rho<1+\delta,
		\]
		then pick any $M>\max\{M'_\rho,M''_\rho\}$
		and set $\alpha_0=\max\big\{\alpha'(M),\alpha''(M)\big\}$.
		For all $\alpha\in(\alpha_0,0)$ one can use both \eqref{eq-ee2} and \eqref{eq-ee3}:
		\[
		E({Q'}^{\alpha,1,\rho}_{\theta,L,\varepsilon,M})\ge -(1+\delta)\alpha^2,\quad
		E({Q''}^{\alpha,1,\rho}_{\theta,L,\varepsilon,M})\ge -(1+\delta)\alpha^2,
		\]
		and the substitution into \eqref{eq-infee} finishes the proof.
	\end{proof}

	\begin{lemma}\label{lem6}
		For any $\theta\in(0,\frac{\pi}{2})$, $L\in(0,\infty)$, $\varepsilon>0$, $M\in (L\tan\theta,\infty)$
		and $\delta>0$ there is $\alpha_1<0$ such that for all $\alpha<\alpha_1$, we have
		\[
		E(\Tilde Q^{\alpha,1}_{\theta,L,\varepsilon,M,1})\ge 
		E(Q^{\alpha,1}_{\theta,L,\varepsilon,M,1})\ge -(1+\delta)\alpha^2.
		\]
	\end{lemma}	
	
	\begin{proof}
		Only the right-hand inequality must be proved, as the left-hand one holds by \eqref{eq-qq}.
		Pick any admissible $(\theta,L,\varepsilon,M)$ and any $\delta>0$. Let $\gamma:(0,\ell)\mapsto \RR^2$ be an arc-length parametrization of $\Gamma_{\theta,L,\varepsilon,M}$
		and denote by $\nu(s)$ the outer unit normal at the point $\gamma(s)$, then the curvature $k(s)$ 
		of $\Gamma_{\theta,L,\varepsilon,M}$ at $\gamma(s)$ is defined by
		$\nu'(s)=k(s)\gamma'(s)$. Remark that $k$ has compact support in our case. For $a>0$ denote
		\[
		U'_a:=\big\{(x,y)\in U_{\theta,L,\varepsilon,M}: \ \text{dist}\big((x,y),\Gamma_{\theta,L,\varepsilon,M}\big)<a\big\}
		\]
		and assume that $a$ is sufficiently small such that the map
		\[
		\Phi:\ (0,\ell)\times (0,a)\ \ni (s,t)\mapsto \gamma(s)-t\nu(s) \in U'_a
		\]
		is a diffeomorphism and $a\|k\|_\infty<1$. For any function $u\in H^1(U'_a)$ denote $v:=u\circ \Phi$,	then a standard computation shows the identities
		\begin{align*}
			\iint_{U'_a} |u|^2\dd x\dd y&=\iint_{(0,\ell)\times(0,a)} \big|v(s,t)\big|^2\big(1-tk(s)\big)\dd s\dd t,\\
			\iint_{U'_a} |\nabla u|^2\dd x\dd y&=\iint_{(0,\ell)\times(0,a)} \Big[\dfrac{1}{1-t k(s)}\big|\partial_s v(s,t)\big|^2\\
			&\quad +\big(1-tk(s)\big)\big|\partial_t v(s,t)\big|^2\Big]\dd s\dd t,\\
			\dint_{\Gamma_{\theta,L,\varepsilon,M}} |u|^2\dd \sigma&=\dint_0^\ell  \big| v(s,0)\big|^2\dd s.
		\end{align*}
		Therefore, for any $u\in \cD(q^{\alpha,1}_{\theta,L,\varepsilon,M})$ one has
		\begin{equation}
			\label{eq-est00}
			\begin{aligned}
				q^{\alpha,1}_{\theta,L,\varepsilon,M}(u,u)&=\iint_{U_{\theta,L,\varepsilon,M}}|\nabla u|^2\dd x\dd y+\alpha\int_{\Gamma_{\theta,L,\varepsilon,M}}|u|^2\dd \sigma\\
				&\ge \iint_{U'_a}|\nabla u|^2\dd x\dd y+\alpha\int_{\Gamma_{\theta,L,\varepsilon,M}}|u|^2\dd \sigma\\
				&\ge \int_0^\ell \bigg[
				\int_0^a \big(1-tk(s)\big)\big|\partial_t v(s,t)\big|^2\dd t+\alpha\big| v(s,0)\big|^2\bigg]\dd s.
			\end{aligned}
		\end{equation}
		
		In \cite[Lem.~5.1]{kop1} it was shown that for
		\[
		\lambda(\alpha,s):=\inf_{f\in H^1(0,a),\, f\ne 0}\dfrac{\dint_0^a \big(1-tk(s)\big)\big|f'(t)\big|^2\dd t+\alpha\big|f(0)\big|^2}{\dint_0^a \big(1-tk(s)\big)\big|f(t)\big|^2\dd t}
		\]
		it holds that $\lambda(\alpha,s)=-\alpha^2+k(s)\alpha+O(\log|\alpha|)$ for $\alpha\to-\infty$,
		and the remainder estimate is uniform in $s$. Therefore, for suitable $C>0$ and $\alpha_*<0$ one
		has $\lambda(\alpha,s)\ge -\alpha^2+C\alpha$ for all $s\in (0,\ell)$ and all $\alpha<\alpha_*$,
		and without loss of generality we additionally assume that $-\alpha^2+C\alpha<0$ for all these $(\alpha,s)$.
		For the same $(\alpha,s)$ the last expression in \eqref{eq-est00} is estimated from below as
		\begin{align*}
			\int_0^\ell
			\int_0^a\bigg[ \big(1-tk(s)\big)\big|&\partial_t v(s,t)\big|^2\dd t+\alpha\big| v(s,0)\big|^2\bigg]\dd s\\
			&\ge \dint_0^\ell\lambda(\alpha,s)\dint_0^a \big(1-tk(s)\big)\big|v(s,t)\big|^2\dd t\dd s\\
			&\ge (-\alpha^2+C\alpha) \dint_0^\ell\dint_0^a \big(1-tk(s)\big)\big|v(s,t)\big|^2\dd t\dd s\\
			&= (-\alpha^2+C\alpha)\iint_{U'_a}|u|^2\dd x\dd y.
		\end{align*}
		Therefore, for all $u\in \cD(q^{\alpha,1}_{\theta,L,\varepsilon,M})$ and all $\alpha<\alpha_*$ we have
		\begin{align*}
			q^{\alpha,1}_{\theta,L,\varepsilon,M}(u,u)&\ge (-\alpha^2+C\alpha)\iint_{U'_a}|u|^2\dd x\dd y\\
			&\ge (-\alpha^2+C\alpha)\iint_{U_{\theta,L,\varepsilon,M}}|u|^2\dd x\dd y,
		\end{align*}
		which yields $E(Q^{\alpha,1}_{\theta,L,\varepsilon,M})\ge -\alpha^2+C\alpha$. Finally, 
		one chooses $\alpha_1<\alpha_*$ such that $|C\alpha|\le \delta \alpha^2$ for all $\alpha<\alpha_1$.
	\end{proof}

	\begin{corollary}\label{corol7}
		Let $\theta\in(0,\frac{\pi}{2})$ and $\delta>0$, then one can find  $L>0$, $\varepsilon>0$, $M>L\tan\theta$
		as well as $\alpha_1\in(-\infty,-1)$ and $\alpha_0\in(-1,0)$ such that
		\begin{align*}
			E(Q^{-1,1}_{\theta,L,\varepsilon,M}), E(\Tilde Q^{-1,1}_{\theta,L,\varepsilon,M})&\in \Big[ -\dfrac{1}{\sin^2\theta},-\dfrac{1}{\sin^2\theta}+\delta\Big],\\
			\dfrac{E(\Tilde Q^{\alpha,1}_{\theta,L,\varepsilon,M})}{\alpha^2}\ge \dfrac{E(Q^{\alpha,1}_{\theta,L,\varepsilon,M})}{\alpha^2}&\ge -(1+\delta) \text{ for all }\alpha\in(\alpha_0,0),\\
			\dfrac{E(\Tilde Q^{\alpha,1}_{\theta,L,\varepsilon,M})}{\alpha^2}\ge \dfrac{E(Q^{\alpha,1}_{\theta,L,\varepsilon,M})}{\alpha^2}&\ge -(1+\delta) \text{ for all }\alpha\in(-\infty,\alpha_1).
		\end{align*}
	\end{corollary}

	\begin{proof}
		Using Lemma \ref{lem4} we choose $\varepsilon$ and $L$ such that the first condition is fulfilled for all $M>L\tan\theta$. Then we choose a sufficiently large $M>0$ and a sufficiently small negative $\alpha_0$ to satisfy the second condition using Lemma~\ref{lem5}. By applying Lemma \ref{lem6} we find a large negative $\alpha_1$  that satisfies the third condition.
	\end{proof}
	
	\section{Proof of Theorem~\ref{thm1}}
	
	We will first construct an unbounded domain $U$ with Lipschitz boundary for which
	the limit $E(R^\alpha_U)/\alpha^2$ for $\alpha\to-\infty$ does not exist. Then we
	construct a required bounded domain $\Omega$ using suitable truncations.
	
	\subsection{Constructing an unbounded ``bad'' domain}\label{ssec31}
	Pick arbitrary
	\[
	0<\theta'<\theta''<\frac{\pi}{2}
	\]
	and denote
	\[
	\delta:=\dfrac{1}{4}\min\bigg\{\Big|\dfrac{1}{\sin^2\theta'}-\dfrac{1}{\sin^2\theta''}\Big|,
	\Big|\dfrac{1}{\sin^2\theta''}-1\Big|\bigg\},
	\]
	then, in particular,
	\[
	-\dfrac{1}{\sin^2\theta'}<-\dfrac{1}{\sin^2\theta'}+\delta<-\dfrac{1}{\sin^2\theta''}<-\dfrac{1}{\sin^2\theta''}+\delta<-(1+\delta).
	\]
	Choose parameters $(L',\varepsilon',M',\alpha'_0,\alpha'_1)$ for the angle $\theta'$, then parameters $(L'',\varepsilon',M'',\alpha''_0,\alpha''_1)$ for the angle $\theta''$, such that the respective estimates of Corollary~\ref{corol7} are satisfied with the above $\delta$, then denote
	\[
	\alpha_0:=\max\{\alpha'_0,\alpha''_0\}\in(-1,0),\quad \alpha_1:=\min\{\alpha'_1,\alpha''_1\}\in(-\infty,-1)
	\]
	and introduce the positive constants
	\[
	t:=\dfrac{\alpha_0}{\alpha_1}<|\alpha_0|<1,\qquad
	\gamma:=\dfrac{1}{t}\equiv \dfrac{\alpha_1}{\alpha_0}>|\alpha_1|>1.
	\]

	For any $m,n\in\NN$ one has:
	\[
	-\gamma^m t^n=-1 \text{ if $m=n$},
	\quad
	-\gamma^m t^n\notin[\alpha_1,\alpha_0] \text{ if $m\ne n$.}
	\]
	Using the inclusions $[\alpha'_1,\alpha'_0]\subset [\alpha_1,\alpha_0]$ and $[\alpha''_1,\alpha''_0]\subset [\alpha_1,\alpha_0]$ as well as the scaling relations \eqref{eq-scale},
	\[
	\dfrac{E(A^{-\gamma^m,t^n}_{\theta',L',\varepsilon',M'})}{(\gamma^m)^2}=
	\dfrac{E(A^{-\gamma^mt^n,1}_{\theta',L',\varepsilon',M'})}{(\gamma^m t^n)^2} \text{ etc.,}
	\]
	we obtain for any $m,n\in\NN$, due to Corollary \ref{corol7},
	\begin{align}
		\label{eq-incl1}	
		\dfrac{E(A^{-\gamma^m,t^n}_{\theta',L',\varepsilon',M'})}{(\gamma^m)^2},\dfrac{E(\Tilde A^{-\gamma^m,t^n}_{\theta',L',\varepsilon',M'})}{(\gamma^m)^2}
		&\in \begin{cases}
			\big[ -\frac{1}{\sin^2\theta'},-\frac{1}{\sin^2\theta'}+\delta\big], & m=n,\\[\medskipamount]
			\big[-(1+\delta),\infty\big), & m\ne n,
		\end{cases}\\
		\label{eq-incl2}	
		\dfrac{E(A^{-\gamma^m,t^n}_{\theta'',L'',\varepsilon'',M''})}{(\gamma^m)^2}, \dfrac{E(\Tilde A^{-\gamma^m,t^n}_{\theta'',L'',\varepsilon'',M''})}{(\gamma^m)^2}&
		\in \begin{cases}
			\big[ -\frac{1}{\sin^2\theta''},-\frac{1}{\sin^2\theta''}+\delta\big], & m=n,\\[\medskipamount]
			\big[-(1+\delta),\infty\big), & m\ne n.
		\end{cases}
	\end{align}		
	
	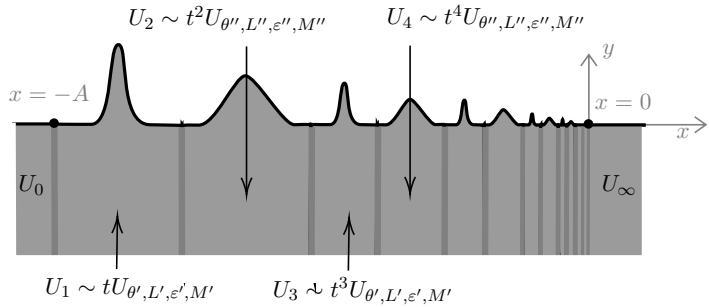
\begin{figure}[b]
		\centering

		\tikzset{every picture/.style={line width=0.75pt}} 
		
		\scalebox{0.8}{\begin{tikzpicture}[x=0.75pt,y=0.75pt,yscale=-1,xscale=1]
				
				\draw [color={rgb, 255:red, 128; green, 128; blue, 128 }  ,draw opacity=1 ] (50,80.75) -- (437,80.75)(364,36.75) -- (364,154) (430,75.75) -- (437,80.75) -- (430,85.75) (359,43.75) -- (364,36.75) -- (369,43.75)  ;
				\draw  [color={rgb, 255:red, 0; green, 0; blue, 0 }  ,draw opacity=1 ][fill={rgb, 255:red, 155; green, 155; blue, 155 }  ,fill opacity=1 ][line width=1.5]  (30.45,80.25) .. controls (32.9,80.5) and (39.25,80.5) .. (53.5,80.5) .. controls (67.75,80.5) and (64,30.4) .. (70.6,30.4) .. controls (77.2,30.4) and (70.2,80.45) .. (90.25,80.75) .. controls (110.3,81.05) and (110,81.25) .. (110.25,81) .. controls (110.5,80.75) and (112,182.25) .. (111.5,177) .. controls (111,171.75) and (30.2,172.75) .. (30.45,175.25) .. controls (30.7,177.75) and (28,80) .. (30.45,80.25) -- cycle ;
				\draw  [color={rgb, 255:red, 0; green, 0; blue, 0 }  ,draw opacity=1 ][fill={rgb, 255:red, 155; green, 155; blue, 155 }  ,fill opacity=1 ][line width=1.5]  (110.25,81) .. controls (110.55,81.5) and (109.75,80.75) .. (120.5,81) .. controls (131.25,81.25) and (143.15,50) .. (149.75,50) .. controls (156.35,50) and (173.25,81) .. (180.5,80.75) .. controls (187.75,80.5) and (190.25,80.5) .. (190.5,80.75) .. controls (190.75,81) and (193.05,183.25) .. (192.75,178) .. controls (192.45,172.75) and (112.55,178.75) .. (111.25,175.75) .. controls (109.95,172.75) and (109.95,80.5) .. (110.25,81) -- cycle ;
				\draw  [color={rgb, 255:red, 0; green, 0; blue, 0 }  ,draw opacity=1 ][fill={rgb, 255:red, 155; green, 155; blue, 155 }  ,fill opacity=1 ][line width=1.5]  (190.79,80.42) .. controls (190.94,80.68) and (195.21,80.68) .. (202.57,80.68) .. controls (209.94,80.68) and (208.13,54.65) .. (211.54,54.65) .. controls (214.95,54.65) and (211.34,80.52) .. (221.7,80.68) .. controls (232.07,80.83) and (231.91,80.94) .. (232.04,80.81) .. controls (232.17,80.68) and (231.5,163.62) .. (231.33,169.67) .. controls (231.16,175.71) and (192.61,172.2) .. (191,178.25) .. controls (189.39,184.3) and (190.63,80.16) .. (190.79,80.42) -- cycle ;
				\draw  [color={rgb, 255:red, 0; green, 0; blue, 0 }  ,draw opacity=1 ][fill={rgb, 255:red, 155; green, 155; blue, 155 }  ,fill opacity=1 ][line width=1.5]  (232.04,80.81) .. controls (232.2,81.07) and (231.78,80.68) .. (237.34,80.81) .. controls (242.9,80.94) and (249.05,64.78) .. (252.46,64.78) .. controls (255.87,64.78) and (264.61,80.81) .. (268.36,80.68) .. controls (272.11,80.55) and (273.4,80.55) .. (273.53,80.68) .. controls (273.66,80.81) and (274,153.33) .. (273.67,169.67) .. controls (273.33,186) and (235,169.33) .. (233,169) .. controls (231,168.67) and (231.89,80.55) .. (232.04,80.81) -- cycle ;
				\draw  [color={rgb, 255:red, 0; green, 0; blue, 0 }  ,draw opacity=1 ][fill={rgb, 255:red, 155; green, 155; blue, 155 }  ,fill opacity=1 ][line width=1.5]  (274.05,80.62) .. controls (274.14,80.77) and (276.67,80.77) .. (281.05,80.77) .. controls (285.42,80.77) and (284.35,65.32) .. (286.38,65.32) .. controls (288.4,65.32) and (286.25,80.68) .. (292.41,80.77) .. controls (298.56,80.87) and (298.47,80.93) .. (298.55,80.85) .. controls (298.63,80.77) and (298.67,149.33) .. (298.33,167) .. controls (298,184.67) and (273.33,173) .. (273.67,169.67) .. controls (274,166.33) and (273.96,80.47) .. (274.05,80.62) -- cycle ;
				\draw  [color={rgb, 255:red, 0; green, 0; blue, 0 }  ,draw opacity=1 ][fill={rgb, 255:red, 155; green, 155; blue, 155 }  ,fill opacity=1 ][line width=1.5]  (298.55,80.85) .. controls (298.64,81) and (298.39,80.77) .. (301.7,80.85) .. controls (305,80.93) and (308.65,71.33) .. (310.68,71.33) .. controls (312.7,71.33) and (317.89,80.85) .. (320.12,80.77) .. controls (322.34,80.7) and (323.11,80.7) .. (323.19,80.77) .. controls (323.26,80.85) and (322.79,161.96) .. (323,167) .. controls (323.21,172.04) and (297.67,167) .. (298.33,167) .. controls (299,167) and (298.46,80.7) .. (298.55,80.85) -- cycle ;
				\draw  [color={rgb, 255:red, 0; green, 0; blue, 0 }  ,draw opacity=1 ][fill={rgb, 255:red, 155; green, 155; blue, 155 }  ,fill opacity=1 ][line width=1.5]  (323.19,80.77) .. controls (323.23,80.84) and (324.34,80.84) .. (326.27,80.84) .. controls (328.2,80.84) and (327.72,74.03) .. (328.62,74.03) .. controls (329.51,74.03) and (328.56,80.8) .. (331.28,80.84) .. controls (333.99,80.88) and (333.95,80.91) .. (333.98,80.88) .. controls (334.01,80.84) and (334.01,153.65) .. (333.53,162.05) .. controls (333.04,170.44) and (324.86,166.71) .. (323.53,168.38) .. controls (322.19,170.05) and (323.15,80.71) .. (323.19,80.77) -- cycle ;
				\draw  [color={rgb, 255:red, 0; green, 0; blue, 0 }  ,draw opacity=1 ][fill={rgb, 255:red, 155; green, 155; blue, 155 }  ,fill opacity=1 ][line width=1.5]  (333.98,80.88) .. controls (334.02,80.94) and (333.91,80.84) .. (335.37,80.88) .. controls (336.82,80.91) and (338.43,76.68) .. (339.32,76.68) .. controls (340.22,76.68) and (342.5,80.88) .. (343.48,80.84) .. controls (344.46,80.81) and (344.8,80.81) .. (344.84,80.84) .. controls (344.87,80.88) and (344.61,155.38) .. (345,165) .. controls (345.39,174.62) and (333.91,163.63) .. (333.75,168.25) .. controls (333.59,172.87) and (333.94,80.81) .. (333.98,80.88) -- cycle ;
				\draw  [color={rgb, 255:red, 0; green, 0; blue, 0 }  ,draw opacity=1 ][fill={rgb, 255:red, 155; green, 155; blue, 155 }  ,fill opacity=1 ][line width=1.5]  (345,80.8) .. controls (345.02,80.83) and (345.56,80.83) .. (346.49,80.83) .. controls (347.43,80.83) and (347.2,77.53) .. (347.63,77.53) .. controls (348.06,77.53) and (347.6,80.81) .. (348.92,80.83) .. controls (350.23,80.85) and (350.21,80.86) .. (350.23,80.85) .. controls (350.24,80.83) and (350.24,160.93) .. (350,165) .. controls (349.76,169.07) and (344.59,169.26) .. (344.75,154.25) .. controls (344.91,139.24) and (344.98,80.77) .. (345,80.8) -- cycle ;
				\draw  [color={rgb, 255:red, 0; green, 0; blue, 0 }  ,draw opacity=1 ][fill={rgb, 255:red, 155; green, 155; blue, 155 }  ,fill opacity=1 ][line width=1.5]  (350.23,80.85) .. controls (350.25,80.88) and (350.19,80.83) .. (350.9,80.85) .. controls (351.6,80.86) and (352.38,78.82) .. (352.82,78.82) .. controls (353.25,78.82) and (354.36,80.85) .. (354.83,80.83) .. controls (355.31,80.81) and (355.47,80.81) .. (355.49,80.83) .. controls (355.5,80.85) and (355.81,162.09) .. (356,166.75) .. controls (356.19,171.41) and (350.08,164.76) .. (350,167) .. controls (349.92,169.24) and (350.21,80.81) .. (350.23,80.85) -- cycle ;
				\draw  [fill={rgb, 255:red, 155; green, 155; blue, 155 }  ,fill opacity=1 ][line width=1.5]  (5.25,80.25) -- (30.45,80.25) -- (30.45,175.25) -- (5.25,175.25) -- cycle ;
				\draw  [fill={rgb, 255:red, 155; green, 155; blue, 155 }  ,fill opacity=1 ][line width=1.5]  (355.49,80.83) -- (398.49,80.83) -- (398.49,168.25) -- (355.49,168.25) -- cycle ;
				\draw [color={rgb, 255:red, 128; green, 128; blue, 128 }  ,draw opacity=1 ][line width=3]    (30.6,81.75) -- (31,179.25) ;
				\draw [color={rgb, 255:red, 128; green, 128; blue, 128 }  ,draw opacity=1 ][line width=3]    (110.1,82.25) -- (110.25,177.75) ;
				\draw [color={rgb, 255:red, 128; green, 128; blue, 128 }  ,draw opacity=1 ][line width=3]    (191.25,82.25) -- (191,178.25) ;
				\draw [color={rgb, 255:red, 128; green, 128; blue, 128 }  ,draw opacity=1 ][line width=3]    (232.5,82) -- (232.25,172) ;
				\draw [color={rgb, 255:red, 128; green, 128; blue, 128 }  ,draw opacity=1 ][line width=3]    (274.25,82) -- (274,172) ;
				\draw [color={rgb, 255:red, 128; green, 128; blue, 128 }  ,draw opacity=1 ][line width=3]    (299.5,82.25) -- (299.25,172.25) ;
				\draw [color={rgb, 255:red, 128; green, 128; blue, 128 }  ,draw opacity=1 ][line width=2.25]    (322.75,82) -- (322.5,172) ;
				\draw [color={rgb, 255:red, 128; green, 128; blue, 128 }  ,draw opacity=1 ][line width=2.25]    (334,82.75) -- (333.75,172.75) ;
				\draw [color={rgb, 255:red, 128; green, 128; blue, 128 }  ,draw opacity=1 ][line width=2.25]    (345,82.75) -- (344.75,172.75) ;
				\draw [color={rgb, 255:red, 128; green, 128; blue, 128 }  ,draw opacity=1 ][line width=2.25]    (355.75,82.25) -- (355.5,172.25) ;
				\draw [color={rgb, 255:red, 128; green, 128; blue, 128 }  ,draw opacity=1 ][line width=2.25]    (350.25,82.5) -- (350,172.5) ;
				\draw [color={rgb, 255:red, 128; green, 128; blue, 128 }  ,draw opacity=1 ][line width=1.5]    (360.25,82) -- (360,172) ;
				\draw [color={rgb, 255:red, 128; green, 128; blue, 128 }  ,draw opacity=1 ][line width=1.5]    (363.75,81.75) -- (363.5,171.75) ;
				\draw [color={rgb, 255:red, 255; green, 255; blue, 255 }  ,draw opacity=1 ][line width=3]    (4.6,79.5) -- (4.9,189.5) ;
				\draw [color={rgb, 255:red, 255; green, 255; blue, 255 }  ,draw opacity=1 ][line width=3]    (398.99,82.08) -- (399,171.75) ;
				\draw  [color={rgb, 255:red, 255; green, 255; blue, 255 }  ,draw opacity=1 ][fill={rgb, 255:red, 255; green, 255; blue, 255 }  ,fill opacity=1 ] (6.75,163.25) -- (400.75,163.25) -- (400.75,179.75) -- (6.75,179.75) -- cycle ;
				\draw    (69.5,171.5) -- (69.97,140) ;
				\draw [shift={(70,138)}, rotate = 90.86] [color={rgb, 255:red, 0; green, 0; blue, 0 }  ][line width=0.75]    (10.93,-3.29) .. controls (6.95,-1.4) and (3.31,-0.3) .. (0,0) .. controls (3.31,0.3) and (6.95,1.4) .. (10.93,3.29)   ;
				\draw    (150.5,31.5) -- (150.5,123) ;
				\draw [shift={(150.5,125)}, rotate = 270] [color={rgb, 255:red, 0; green, 0; blue, 0 }  ][line width=0.75]    (10.93,-3.29) .. controls (6.95,-1.4) and (3.31,-0.3) .. (0,0) .. controls (3.31,0.3) and (6.95,1.4) .. (10.93,3.29)   ;
				\draw    (214,172.5) -- (214.47,141) ;
				\draw [shift={(214.5,139)}, rotate = 90.86] [color={rgb, 255:red, 0; green, 0; blue, 0 }  ][line width=0.75]    (10.93,-3.29) .. controls (6.95,-1.4) and (3.31,-0.3) .. (0,0) .. controls (3.31,0.3) and (6.95,1.4) .. (10.93,3.29)   ;
				\draw    (252.5,31) -- (252.5,122.5) ;
				\draw [shift={(252.5,124.5)}, rotate = 270] [color={rgb, 255:red, 0; green, 0; blue, 0 }  ][line width=0.75]    (10.93,-3.29) .. controls (6.95,-1.4) and (3.31,-0.3) .. (0,0) .. controls (3.31,0.3) and (6.95,1.4) .. (10.93,3.29)   ;
				\draw  [fill={rgb, 255:red, 0; green, 0; blue, 0 }  ,fill opacity=1 ] (361.25,80.25) .. controls (361.25,78.87) and (362.37,77.75) .. (363.75,77.75) .. controls (365.13,77.75) and (366.25,78.87) .. (366.25,80.25) .. controls (366.25,81.63) and (365.13,82.75) .. (363.75,82.75) .. controls (362.37,82.75) and (361.25,81.63) .. (361.25,80.25) -- cycle ;
				\draw  [fill={rgb, 255:red, 0; green, 0; blue, 0 }  ,fill opacity=1 ] (27.95,79.75) .. controls (27.95,78.37) and (29.07,77.25) .. (30.45,77.25) .. controls (31.83,77.25) and (32.95,78.37) .. (32.95,79.75) .. controls (32.95,81.13) and (31.83,82.25) .. (30.45,82.25) .. controls (29.07,82.25) and (27.95,81.13) .. (27.95,79.75) -- cycle ;
				
				\draw (23.45,175.65) node [anchor=north west][inner sep=0.75pt]    {$U_{1} \sim tU_{\theta ',L',\varepsilon ',M'}$};
				\draw (75.45,4.65) node [anchor=north west][inner sep=0.75pt]    {$U_{2} \sim t^{2} U_{\theta '',L'',\varepsilon '',M''}$};
				\draw (165.45,175.15) node [anchor=north west][inner sep=0.75pt]    {$U_{3} \sim t^{3} U_{\theta ',L',\varepsilon ',M'}$};
				\draw (236.45,3.9) node [anchor=north west][inner sep=0.75pt]    {$U_{4} \sim t^{4} U_{\theta '',L'',\varepsilon '',M''}$};
				\draw (417.25,81.4) node [anchor=north west][inner sep=0.75pt]  [color={rgb, 255:red, 128; green, 128; blue, 128 }  ,opacity=1 ]  {$x$};
				\draw (371,28.4) node [anchor=north west][inner sep=0.75pt]  [color={rgb, 255:red, 128; green, 128; blue, 128 }  ,opacity=1 ]  {$y$};
				\draw (366.25,60.15) node [anchor=north west][inner sep=0.75pt]  [color={rgb, 255:red, 128; green, 128; blue, 128 }  ,opacity=1 ]  {$x=0$};
				\draw (0.75,54.65) node [anchor=north west][inner sep=0.75pt]  [color={rgb, 255:red, 128; green, 128; blue, 128 }  ,opacity=1 ]  {$x=-A$};
				\draw (6.75,110.9) node [anchor=north west][inner sep=0.75pt]    {$U_{0}$};
				\draw (370.5,110.9) node [anchor=north west][inner sep=0.75pt]    {$U_{\infty }$};

		\end{tikzpicture}}

		\caption{The structure of the domain $U$. The subdomains $U_n$ are glued along the side boundaries shown as bold dark gray lines.}\label{fig:uu}
	\end{figure}
	
	For $n\in\NN$ consider the domains
	\[
	U_n:=\begin{cases}
		t^n U_{\theta',L',\varepsilon',M'}, & \text{$n$ is odd,}\\
		t^n U_{\theta'',L'',\varepsilon'',M''}, & \text{$n$ is even,}
	\end{cases}
	\]
	then put them next to each other and glue along side boundaries such that the left boundary corresponds to $x=-A$ for some $A>0$ and the right boundary corresponds to $x=0$.
	The domain is then completed by gluing it with $U_0:=(-\infty,-A)\times (-\infty,0)$ on the left  and with $U_\infty:=(0,\infty)\times(-\infty,0)$ on the right. The constructions are illustrated in Fig.~\ref{fig:uu}. The resulting domain $U$
	has the form
	\[
	U:=\big\{(x,y)\in\RR^2:\, y<h(x)\big\}
	\]
	for a continuous function $h:\RR\to\RR$. If the transition from $U_{n-1}$ to $U_n$ happens along the line $x=a_n$, then on the interval $(a_{n},a_{n+1})$ the function $h$ coincides either with $t^nh_{\theta',L',\varepsilon',M'}(t^{-n}\cdot-b_n)$ or with $t^n h_{\theta'',L'',\varepsilon'',M''}(t^{-n}\cdot-b_n)$
	with some $b_n\in\RR$, and $h$ is identically zero outside $(-A,0)$. It follows that
	\[
	\|h'\|_\infty\le \max\big\{\|h'_{\theta',L',\varepsilon',M'}\|_\infty,\|h'_{\theta'',L'',\varepsilon'',M''}\|_\infty\big\}<\infty,
	\]
	in particular, $h$  is Lipschitz. In fact, the above discussion
	even shows that $h$ is $C^\infty$ at all points except at $0$.
	
	We are interested in the associated Robin Laplacian $R^\alpha_U$.
	By using the standard Dirichlet-Neumann bracketing along the gluing lines one shows 
	\[
	E\Big( A^\alpha_0 \oplus \bigoplus_{n=1}^\infty A^\alpha_n \oplus A^\alpha_\infty\Big)
	\le
	E(R^\alpha_U)\le 
	E\Big( \Tilde A^\alpha_0 \oplus \bigoplus_{n=1}^\infty \Tilde A^\alpha_n \oplus \Tilde A^\alpha_\infty\Big)
	\]
	where $A^\alpha_n$ (respectively $\Tilde A^\alpha_n$) are the Laplacians in $U_n$ with the Robin boundary condition $\partial_\nu u+\alpha u=0$ along the curve $y=h(x)$ (the top boundary),
	and Neumann (respectively Dirichlet) boundary conditions on the side boundaries.
	
	The operators $A^\alpha_0,\Tilde A^\alpha_0, A^\alpha_\infty,\Tilde A^\alpha_\infty$ admit a separation of variables, and for any $\alpha<0$ one has $E(A^\alpha_0)=E(\Tilde A^\alpha_0)=E(A^\alpha_\infty)=E(\Tilde A^\alpha_\infty)=-\alpha^2$.	On top of that, for any $n\in\NN$ the following holds:
	\begin{itemize}
		\item[(i)] $A_n^\alpha$ (respectively $\Tilde A_n^\alpha$) is unitarily equivalent to $A^{\alpha,t^n}_{\theta',L',\varepsilon',M'}$ (respectively $\Tilde A^{\alpha,t^n}_{\theta',L',\varepsilon',M'}$), if $n$ is odd,
		\item[(ii)] $A_n^\alpha$ (respectively $\Tilde A_n^\alpha$) is unitarily equivalent to $A^{\alpha,t^n}_{\theta'',L'',\varepsilon',M''}$ (respectively $\Tilde A^{\alpha,t^n}_{\theta'',L'',\varepsilon'',M''}$), if $n$ is even.
	\end{itemize}
	Therefore, for any $\alpha<0$ we have
	\begin{align*}
		E\Big( A^\alpha_0 \oplus &\bigoplus_{n=1}^\infty A^\alpha_n \oplus A^\alpha_\infty\Big)
		=\min\Big\{E(A^\alpha_0), E(A^\alpha_\infty), \inf_{n\in\NN} E(A_n^\alpha)\Big\},\\
		&=\min\Big\{-\alpha^2,\inf_{p\in\NN} E(A^{\alpha,t^{2p-1}}_{\theta',L',\varepsilon',M'}),
		\inf_{p\in\NN} E(A^{\alpha,t^{2p}}_{\theta'',L'',\varepsilon'',M''})\Big\},\\
		\intertext{and analogously}
		E\Big( \Tilde A^\alpha_0 \oplus &\bigoplus_{n=1}^\infty \Tilde A^\alpha_n \oplus \Tilde A^\alpha_\infty\Big)
		=\min\Big\{E(\Tilde A^\alpha_0), E(\Tilde A^\alpha_\infty), \inf_{n\in\NN} E(\Tilde A_n^\alpha)\Big\},\\
		&=\min\Big\{-\alpha^2,\inf_{p\in\NN} E(\Tilde A^{\alpha,t^{2p-1}}_{\theta',L',\varepsilon',M'}),
		\inf_{p\in\NN} E(\Tilde A^{\alpha,t^{2p}}_{\theta'',L'',\varepsilon'',M''})\Big\}.
	\end{align*}	
	To summarize, for any $\alpha<0$, we have
	\begin{equation}
		\label{eq-twoside1}
		\begin{aligned}
			\min\Big\{-1,&\inf_{p\in\NN} \dfrac{E(A^{\alpha,t^{2p-1}}_{\theta',L',\varepsilon',M'})}{\alpha^2},
			\inf_{p\in\NN} \dfrac{E(A^{\alpha,t^{2p}}_{\theta'',L'',\varepsilon'',M''})}{\alpha^2}\Big\}	\le \dfrac{E(R_U^\alpha)}{\alpha^2}\\
			&\le \min\Big\{-1,\inf_{p\in\NN} \dfrac{E(\Tilde A^{\alpha,t^{2p-1}}_{\theta',L',\varepsilon',M'})}{\alpha^2},
			\inf_{p\in\NN} \dfrac{E(\Tilde A^{\alpha,t^{2p}}_{\theta'',L'',\varepsilon'',M''})}{\alpha^2}\Big\}.
		\end{aligned}
	\end{equation}
	
	Let us substitute $\alpha:=-\gamma^{2q-1}$ with $q\in\NN$. By \eqref{eq-incl1} and \eqref{eq-incl2} we obtain
	\begin{align*}
		\dfrac{E(A^{-\gamma^{2q-1},t^{2q-1}}_{\theta',L',\varepsilon',M'})}{(\gamma^{2q-1})^2},
		\dfrac{E(\Tilde A^{-\gamma^{2q-1},t^{2q-1}}_{\theta',L',\varepsilon',M'})}{(\gamma^{2q-1})^2}&\in \Big[-\dfrac{1}{\sin^2\theta'},-\dfrac{1}{\sin^2\theta'}+\delta\Big],\\
		\dfrac{E(A^{-\gamma^{2q-1},t^{2p-1}}_{\theta',L',\varepsilon',M'})}{(\gamma^{2q-1})^2},
		\dfrac{E(\Tilde A^{-\gamma^{2q-1},t^{2p-1}}_{\theta',L',\varepsilon',M'})}{(\gamma^{2q-1})^2}&\ge -(1+\delta) \text{ for all }p\ne q,\\
		\dfrac{E(A^{-\gamma^{2q-1},t^{2p}}_{\theta'',L'',\varepsilon'',M''})}{(\gamma^{2q-1})^2},
		\dfrac{E(\Tilde A^{-\gamma^{2q-1},t^{2p}}_{\theta'',L'',\varepsilon'',M''})}{(\gamma^{2q-1})^2}&\ge -(1+\delta) \text{ for all }p,
	\end{align*}
	hence, the both sides of \eqref{eq-twoside1} are in $\big[-\frac{1}{\sin^2\theta'},-\frac{1}{\sin^2\theta'}+\delta\big]$, which
	implies
	\begin{equation*}
		\dfrac{E(R_U^{-\gamma^{2q-1}})}{(\gamma^{2q-1})^2}\in \Big[-\dfrac{1}{\sin^2\theta'},-\dfrac{1}{\sin^2\theta'}+\delta\Big]
		\text{ for all }q\in\NN.
	\end{equation*}
	Simlarly, by considering $\alpha:=-\gamma^{2q}$ with $q\in\NN$ and using again \eqref{eq-incl1} and \eqref{eq-incl2} we show
	\begin{align*}
		\dfrac{E(A^{-\gamma^{2q},t^{2p-1}}_{\theta',L',\varepsilon',M'})}{(\gamma^{2q})^2},
		\dfrac{E(\Tilde A^{-\gamma^{2q},t^{2p-1}}_{\theta',L',\varepsilon',M'})}{(\gamma^{2q})^2}&\ge -(1+\delta) \text{ for all }p,\\
		\dfrac{E(A^{-\gamma^{2q},t^{2q}}_{\theta'',L'',\varepsilon'',M''})}{(\gamma^{2q})^2},
		\dfrac{E(\Tilde A^{-\gamma^{2q},t^{2q}}_{\theta'',L'',\varepsilon'',M''})}{(\gamma^{2q})^2}&\in \Big[-\dfrac{1}{\sin^2\theta''},-\dfrac{1}{\sin^2\theta''}+\delta\Big],\\
		\dfrac{E(A^{-\gamma^{2q},t^{2p}}_{\theta'',L'',\varepsilon'',M''})}{(\gamma^{2q})^2},
		\dfrac{E(\Tilde A^{-\gamma^{2q},t^{2p}}_{\theta'',L'',\varepsilon'',M''})}{(\gamma^{2q})^2}
		&\ge -(1+\delta) \text{ for all }p\ne q,
	\end{align*}
	which yields
	\begin{equation*}
		\dfrac{E(R_U^{-\gamma^{2q}})}{(\gamma^{2q})^2}\in \Big[-\dfrac{1}{\sin^2\theta''},-\dfrac{1}{\sin^2\theta''}+\delta\Big]
		\text{ for all }q\in\NN.
	\end{equation*}
	
	Therefore, for the sequences $(\gamma'_k)_{k\in\NN}$ and $(\gamma''_k)_{k\in\NN}$,
	\begin{equation}
		\label{eq-beta}
		\gamma'_k:=-\gamma^{2k-1},\quad \gamma''_k:=-\gamma^{2k},
	\end{equation}
	we have $\gamma'_k\xrightarrow{k\to\infty}-\infty$, $\gamma''_k\xrightarrow{k\to\infty}-\infty$, and for all $k\in\NN$, 
	\begin{equation}
		\label{eq-uu1}
		\begin{aligned}
			\dfrac{E(R_U^{\gamma'_k})}{(\gamma'_k)^2}&\in \Big[-\dfrac{1}{\sin^2\theta'},-\dfrac{1}{\sin^2\theta'}+\delta\Big],\\
			\quad \dfrac{E(R_U^{\gamma''_k})}{(\gamma''_k)^2}&\in \Big[-\dfrac{1}{\sin^2\theta''},-\dfrac{1}{\sin^2\theta''}+\delta\Big].
		\end{aligned}
	\end{equation}
	By construction, the intervals on the right-hand sides of \eqref{eq-uu1} are disjoint, so we are quite close to satisfying the requirements of Theorem~\ref{thm1}, but the domain $U$ is unbounded.
	
	\subsection{Constructing a bounded ``bad'' domain}\label{ssec32}
	By construction, the unbounded domain $U$ coincides with the half-plane
	$\Tilde U:=\RR\times(-\infty,0)$
	outside a compact set, so we can choose an open ball $B$ with $U\setminus \overline{B}=\Tilde U\setminus \overline{B}$. Then we pick another larger open ball $B'$ with $\overline{B}\subset B'$, see Fig.~\ref{fig:trunc}.

	\begin{figure}[b]
	\centering

	\scalebox{0.5}{\begin{tikzpicture}[x=0.75pt,y=0.75pt,yscale=-1,xscale=1]
		
		\draw  [color={rgb, 255:red, 155; green, 155; blue, 155 }  ,draw opacity=1 ][fill={rgb, 255:red, 155; green, 155; blue, 155 }  ,fill opacity=1 ][line width=0.75]  (21.09,137.35) .. controls (22.69,136.55) and (61.78,138.01) .. (77.83,137.68) .. controls (93.89,137.35) and (89.09,118.55) .. (93.2,118.6) .. controls (97.32,118.65) and (94.29,138.55) .. (108.38,137.97) .. controls (122.47,137.39) and (114.29,126.55) .. (123.5,126.1) .. controls (132.72,125.65) and (131.73,138.14) .. (139.21,137.75) .. controls (146.69,137.35) and (143.56,130.28) .. (147.89,130.15) .. controls (152.22,130.01) and (149.78,137.24) .. (157.03,137.89) .. controls (164.29,138.55) and (159.34,131.95) .. (162.82,131.76) .. controls (166.3,131.57) and (162.72,137.99) .. (168.91,137.84) .. controls (175.09,137.7) and (173.84,133.95) .. (176.09,133.7) .. controls (178.34,133.45) and (177.97,137.93) .. (181.67,137.91) .. controls (185.36,137.9) and (181.89,133.97) .. (185.1,134.27) .. controls (188.32,134.57) and (184.97,137.88) .. (188.72,137.88) .. controls (192.47,137.88) and (307.83,137.93) .. (307.84,137.95) .. controls (307.84,137.96) and (299,226) .. (302.5,243) .. controls (306,260) and (23,252) .. (22,251) .. controls (21,250) and (21.81,230.03) .. (22,195.5) .. controls (22.19,160.97) and (19.49,138.15) .. (21.09,137.35) -- cycle ;
		\draw  [draw opacity=0][fill={rgb, 255:red, 155; green, 155; blue, 155 }  ,fill opacity=1 ] (201.99,141.4) .. controls (202,141.68) and (202,141.96) .. (202,142.25) .. controls (202,175.53) and (172,202.5) .. (135,202.5) .. controls (98,202.5) and (68,175.53) .. (68,142.25) -- (135,142.25) -- cycle ;
		\draw  [color={rgb, 255:red, 128; green, 128; blue, 128 }  ,draw opacity=1 ][line width=2.25]  (44.5,132) .. controls (44.5,81.47) and (85.47,40.5) .. (136,40.5) .. controls (186.53,40.5) and (227.5,81.47) .. (227.5,132) .. controls (227.5,182.53) and (186.53,223.5) .. (136,223.5) .. controls (85.47,223.5) and (44.5,182.53) .. (44.5,132) -- cycle ;
		\draw  [dash pattern={on 2.53pt off 3.02pt}][line width=2.25]  (72.75,132) .. controls (72.75,97.07) and (101.07,68.75) .. (136,68.75) .. controls (170.93,68.75) and (199.25,97.07) .. (199.25,132) .. controls (199.25,166.93) and (170.93,195.25) .. (136,195.25) .. controls (101.07,195.25) and (72.75,166.93) .. (72.75,132) -- cycle ;
		\draw [color={rgb, 255:red, 0; green, 0; blue, 0 }  ,draw opacity=1 ][line width=2.25]    (21.09,137.35) .. controls (22.69,136.55) and (61.78,138.01) .. (77.83,137.68) .. controls (93.89,137.35) and (89.09,118.55) .. (93.2,118.6) .. controls (97.32,118.65) and (94.29,138.55) .. (108.38,137.97) .. controls (122.47,137.39) and (114.29,126.55) .. (123.5,126.1) .. controls (132.72,125.65) and (131.73,138.14) .. (139.21,137.75) .. controls (146.69,137.35) and (143.56,130.28) .. (147.89,130.15) .. controls (152.22,130.01) and (149.78,137.24) .. (157.03,137.89) .. controls (164.29,138.55) and (159.34,131.95) .. (162.82,131.76) .. controls (166.3,131.57) and (162.72,137.99) .. (168.91,137.84) .. controls (175.09,137.7) and (173.84,133.95) .. (176.09,133.7) .. controls (178.34,133.45) and (177.97,137.93) .. (181.67,137.91) .. controls (185.36,137.9) and (181.89,133.97) .. (185.1,134.27) .. controls (188.32,134.57) and (184.97,137.88) .. (188.72,137.88) .. controls (192.47,137.88) and (307.83,137.93) .. (307.84,137.95) ;
		\draw  [draw opacity=0][fill={rgb, 255:red, 155; green, 155; blue, 155 }  ,fill opacity=1 ] (178,218.5) -- (251,218.5) -- (251,244.5) -- (178,244.5) -- cycle ;
		
		\draw (142.5,78.9) node [anchor=north west][inner sep=0.75pt]  [font=\LARGE]  {$B$};
		\draw (258,172.9) node [anchor=north west][inner sep=0.75pt]  [font=\LARGE]  {$U$};
		\draw (226,64.9) node [anchor=north west][inner sep=0.75pt]  [font=\LARGE,color={rgb, 255:red, 128; green, 128; blue, 128 }  ,opacity=1 ]  {$B'$};

	\end{tikzpicture}}

		\caption{Truncations of $U$.}\label{fig:trunc}
	\end{figure}
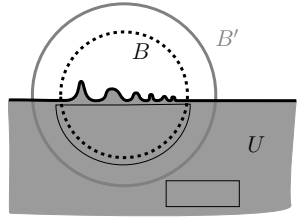
	
	Denote by $T^\alpha$ the self-adjoint operator in $L^2(B'\cap U)$ given by the hermitian sesquilinear form $t^\alpha$ defined by
	\begin{align*}
		t^\alpha(u,u)&:=\iint_{B'\cap U} |\nabla u|^2\dd x\dd y+\alpha\dint_{B'\cap \partial U}|u|^2\dd\sigma,\\
		\quad\cD(t^\alpha)&:=\{u\in H^1(B'\cap U):\, u=0 \text{ on } \partial B'\}.
	\end{align*}
	Any function in $\cD(t^\alpha)$ can be extended by zero to a function in $\cD(r^\alpha_U)$,
	then the min-max-principle implies
	\begin{equation}
		\label{eq-lup}
		E(R_U^\alpha)\le E(T^\alpha) \text{ for any }\alpha\in\RR.
	\end{equation}
	Now let us pick two $C^\infty$ functions $\chi,\Tilde\chi:\RR^2\to\RR$ such that
	\[
	\chi^2+\Tilde\chi^2=1,\quad \supp \chi\subset B',\quad \supp \Tilde\chi\subset \RR^2\setminus \overline{B}. 
	\]
	For any $u\in \cD(r^\alpha_U)$ one has the inclusions
	\[
	\supp \chi u\subset B'\cap U, \quad \supp\Tilde\chi u\subset U\setminus \overline{B}\equiv \Tilde U\setminus \overline{B}\subset \Tilde U,
	\]
	therefore,
	\begin{align*}
		r^\alpha_U(u,u)&
		=\iint_{U} |\nabla (\chi u)|^2\dd x\dd y+\alpha \int_{\partial U}|\chi u|^2\dd\sigma\\
		&\quad +\iint_{U} |\nabla (\Tilde \chi u)|^2\dd x\dd y+\alpha \int_{\partial U}|\Tilde\chi u|^2\dd\sigma\\
		&\quad - \iint_{U} \big(|\nabla \chi|^2 +|\nabla \Tilde \chi|^2\big)|u|^2\dd x\dd y\\
		&=t^\alpha(\chi u,\chi u)+r^\alpha_{\Tilde U}(\Tilde\chi u,\Tilde \chi u)- \iint_{U} \big(|\nabla \chi|^2 +|\nabla \Tilde \chi|^2\big)|u|^2\dd x\dd y\\
		&\ge t^\alpha(\chi u,\chi u) +r^\alpha_{\Tilde U}(\Tilde\chi u,\Tilde \chi u)-C\|u\|^2_{L^2(U)}
	\end{align*}
	with $C:=\big\||\nabla \chi|^2 +|\nabla \Tilde \chi|^2\big\|_\infty$.
	The min-max principle implies
\[
E(R_U^\alpha)\ge E(T^\alpha\oplus R^\alpha_{\Tilde U})-C=\min\big\{E(T^\alpha),E(R^\alpha_{\Tilde U})\big\}-C.
\]
The operator $R^\alpha_{\Tilde U}$ admits a separation of variables, and with the help of Lemma \ref{lem-halfline} for any $\alpha<0$, one obtains $E(R^\alpha_{\Tilde U})=-\alpha^2$, which yields
\begin{equation}
\label{eq-llow1}
	E(R_U^\alpha)\ge \min\big\{E(T^\alpha),-\alpha^2\big\}-C.
\end{equation}
For $\gamma_k\in\{\gamma'_k,\gamma''_k\}$ due to \eqref{eq-beta} and \eqref{eq-uu1} one can find a large $N\in\NN$ such that
	\begin{equation}
		\label{eq-lll}
		E(R_U^{\gamma_k})\le \Big(-\dfrac{1}{\sin^2\theta''}+\delta\Big)\gamma_k^2\le -\gamma_k^2-C \text{ for all } k>N,
	\end{equation}
	then \eqref{eq-llow1} implies $E(R_U^{\gamma_k})\ge E(T^{\gamma_k})-C$, and by combining with \eqref{eq-lup} we arrive at
	\begin{equation}
		\label{eq-twoside2}
		E(R_U^{\gamma_k})\le E(T^{\gamma_k})\le E(R_U^{\gamma_k})+C \text{ for all $k>N$ and $\gamma_k\in\{\gamma'_k,\gamma''_k\}$.}
	\end{equation}
	
	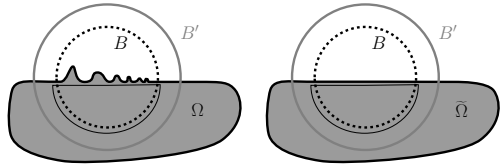
\begin{figure}[b]
		\centering
		
		\scalebox{0.4}{\begin{tikzpicture}[x=0.75pt,y=0.75pt,yscale=-1,xscale=1]
				
				\draw  [color={rgb, 255:red, 0; green, 0; blue, 0 }  ,draw opacity=1 ][fill={rgb, 255:red, 155; green, 155; blue, 155 }  ,fill opacity=1 ][line width=2.25]  (37.5,167.75) .. controls (40.75,157.5) and (52.47,156.97) .. (64.34,157.23) .. controls (76.22,157.49) and (89.8,157.85) .. (97.83,157.68) .. controls (105.86,157.51) and (109.09,138.55) .. (113.2,138.6) .. controls (117.32,138.65) and (114.29,158.55) .. (128.38,157.97) .. controls (142.47,157.39) and (134.29,146.55) .. (143.5,146.1) .. controls (152.72,145.65) and (151.73,158.14) .. (159.21,157.75) .. controls (166.69,157.35) and (163.56,150.28) .. (167.89,150.15) .. controls (172.22,150.01) and (169.78,157.24) .. (177.03,157.89) .. controls (184.29,158.55) and (179.34,151.95) .. (182.82,151.76) .. controls (186.3,151.57) and (182.72,157.99) .. (188.91,157.84) .. controls (195.09,157.7) and (193.84,153.95) .. (196.09,153.7) .. controls (198.34,153.45) and (197.97,157.93) .. (201.67,157.91) .. controls (205.36,157.9) and (201.89,153.97) .. (205.1,154.27) .. controls (208.32,154.57) and (204.97,157.88) .. (208.72,157.88) .. controls (212.47,157.88) and (240.37,157.89) .. (269.68,157.91) .. controls (298.99,157.92) and (307.75,157.75) .. (316,168.75) .. controls (324.25,179.75) and (328.5,208.25) .. (299.75,230.5) .. controls (271,252.75) and (134,262.5) .. (100,257.25) .. controls (66,252) and (39.5,241.5) .. (35.25,223.25) .. controls (31,205) and (34.25,178) .. (37.5,167.75) -- cycle ;
				\draw  [draw opacity=0][fill={rgb, 255:red, 155; green, 155; blue, 155 }  ,fill opacity=1 ] (221.99,161.4) .. controls (222,161.68) and (222,161.96) .. (222,162.25) .. controls (222,195.53) and (192,222.5) .. (155,222.5) .. controls (118,222.5) and (88,195.53) .. (88,162.25) -- (155,162.25) -- cycle ;
				\draw  [color={rgb, 255:red, 128; green, 128; blue, 128 }  ,draw opacity=1 ][line width=2.25]  (64.5,152) .. controls (64.5,101.47) and (105.47,60.5) .. (156,60.5) .. controls (206.53,60.5) and (247.5,101.47) .. (247.5,152) .. controls (247.5,202.53) and (206.53,243.5) .. (156,243.5) .. controls (105.47,243.5) and (64.5,202.53) .. (64.5,152) -- cycle ;
				\draw  [dash pattern={on 2.53pt off 3.02pt}][line width=2.25]  (92.75,152) .. controls (92.75,117.07) and (121.07,88.75) .. (156,88.75) .. controls (190.93,88.75) and (219.25,117.07) .. (219.25,152) .. controls (219.25,186.93) and (190.93,215.25) .. (156,215.25) .. controls (121.07,215.25) and (92.75,186.93) .. (92.75,152) -- cycle ;
				\draw  [color={rgb, 255:red, 0; green, 0; blue, 0 }  ,draw opacity=1 ][fill={rgb, 255:red, 155; green, 155; blue, 155 }  ,fill opacity=1 ][line width=2.25]  (359.5,167.75) .. controls (362.75,157.5) and (374.47,156.97) .. (386.34,157.23) .. controls (398.22,157.49) and (411.8,157.85) .. (419.83,157.68) .. controls (427.86,157.51) and (587.57,157.86) .. (591.68,157.91) .. controls (595.8,157.96) and (629.75,157.75) .. (638,168.75) .. controls (646.25,179.75) and (650.5,208.25) .. (621.75,230.5) .. controls (593,252.75) and (456,262.5) .. (422,257.25) .. controls (388,252) and (361.5,241.5) .. (357.25,223.25) .. controls (353,205) and (356.25,178) .. (359.5,167.75) -- cycle ;
				\draw  [draw opacity=0][fill={rgb, 255:red, 155; green, 155; blue, 155 }  ,fill opacity=1 ] (543.99,161.4) .. controls (544,161.68) and (544,161.96) .. (544,162.25) .. controls (544,195.53) and (514,222.5) .. (477,222.5) .. controls (440,222.5) and (410,195.53) .. (410,162.25) -- (477,162.25) -- cycle ;
				\draw  [color={rgb, 255:red, 128; green, 128; blue, 128 }  ,draw opacity=1 ][line width=2.25]  (386.5,152) .. controls (386.5,101.47) and (427.47,60.5) .. (478,60.5) .. controls (528.53,60.5) and (569.5,101.47) .. (569.5,152) .. controls (569.5,202.53) and (528.53,243.5) .. (478,243.5) .. controls (427.47,243.5) and (386.5,202.53) .. (386.5,152) -- cycle ;
				\draw  [dash pattern={on 2.53pt off 3.02pt}][line width=2.25]  (414.75,152) .. controls (414.75,117.07) and (443.07,88.75) .. (478,88.75) .. controls (512.93,88.75) and (541.25,117.07) .. (541.25,152) .. controls (541.25,186.93) and (512.93,215.25) .. (478,215.25) .. controls (443.07,215.25) and (414.75,186.93) .. (414.75,152) -- cycle ;
				
				\draw (162.5,98.9) node [anchor=north west][inner sep=0.75pt]  [font=\LARGE]  {$B$};
				\draw (246,84.9) node [anchor=north west][inner sep=0.75pt]  [font=\LARGE,color={rgb, 255:red, 128; green, 128; blue, 128 }  ,opacity=1 ]  {$B'$};
				\draw (484.5,98.9) node [anchor=north west][inner sep=0.75pt]  [font=\LARGE]  {$B$};
				\draw (568,84.9) node [anchor=north west][inner sep=0.75pt]  [font=\LARGE,color={rgb, 255:red, 128; green, 128; blue, 128 }  ,opacity=1 ]  {$B'$};
				\draw (260,184.4) node [anchor=north west][inner sep=0.75pt]  [font=\LARGE]  {$\Omega $};
				\draw (589,180.9) node [anchor=north west][inner sep=0.75pt]  [font=\LARGE]  {$\Tilde\Omega $};

			\end{tikzpicture}
		}
		
		\caption{The bounded domains $\Omega$ (Lipschitz) and $\Tilde \Omega$ ($C^\infty$ smooth).}\label{fig:omega}
	\end{figure}	 
	
	Now let us choose a bounded domain $\Omega\subset\RR^2$ such that
	$B'\cap \Omega= B'\cap U$ 	and the boundary $\partial\Omega$ is $C^\infty$ outside $B'$, see Fig.~\ref{fig:omega},
	then $\Omega$ is a bounded Lipschitz domain. Furthermore, the set
	$\Tilde \Omega := (\Omega\setminus B') \,\cup\,(\Tilde U\cap B')$
	obtained by replacing the non-smooth portion $B'\cap \Omega$ of $\Omega$ with $B'\cap \Tilde U$ is
	a bounded $C^\infty$ domain, and	$\Omega\setminus \overline{B}=\Tilde \Omega\setminus \overline{B}$.
	
	As each function in $\cD(t^\alpha)$ can be extended by zero to a function in $\cD(r^\alpha_\Omega)$, the min-max principle yields
	\begin{equation}
		\label{eq-lup2}
		E(R_\Omega^\alpha)\le E(T^\alpha) \text{ for any }\alpha\in\RR.
	\end{equation}
	Using the same functions $\chi$ and $\Tilde \chi$ as above,
	for any $u\in \cD(r^\alpha_\Omega)$ we obtain
	\[
	\supp\chi u\subset B'\cap \Omega\equiv B'\cap U,
	\quad
	\supp\Tilde\chi u\subset \Omega\setminus\overline{B}\equiv \Tilde\Omega\setminus\overline{B}\subset\Tilde\Omega,
	\]
	therefore,
	\begin{align*}
		r^\alpha_\Omega(u,u)&=\iint_{\Omega} |\nabla (\chi u)|^2\dd x\dd y+\alpha \int_{\partial \Omega}|\chi u|^2\dd\sigma\\
		&\quad +\iint_{\Omega} |\nabla (\Tilde \chi u)|^2\dd x\dd y+\alpha \int_{\partial \Omega}|\Tilde\chi u|^2\dd\sigma\\
		&\quad- \iint_{\Omega} \big(|\nabla \chi|^2 +|\nabla \Tilde \chi|^2\big)|u|^2\dd x\dd y\\
		&=t^\alpha(\chi u,\chi u)+ r^\alpha_{\Tilde\Omega}(\Tilde\chi u,\Tilde\chi u)
		- \iint_{\Omega} \big(|\nabla \chi|^2 +|\nabla \Tilde \chi|^2\big)|u|^2\dd x\dd y\\
		&\ge t^\alpha(\chi u,\chi u) + r^\alpha_{\Tilde\Omega}(\Tilde\chi u,\Tilde\chi u)
		-C\|u\|^2_{L^2(\Omega)},
	\end{align*}
	and it follows by the min-max principle that
	\begin{equation}
		\label{eq-eee2}
		\begin{aligned}
			E(R_\Omega^\alpha)&\ge E(T^\alpha\oplus R^\alpha_{\Tilde \Omega})-C\equiv \min\big\{E(T^\alpha),E(R^\alpha_{\Tilde \Omega})\big\}-C.
		\end{aligned}
	\end{equation}
	One has $E(R_{\Tilde \Omega}^\alpha)=-\alpha^2+o(\alpha^2)$ for $\alpha\to -\infty$ as discussed in the introduction, and due to \eqref{eq-lll} and \eqref{eq-twoside2} we can increase the above value of $N$ to have
	\[
	E(T^{\gamma_k})\le E(R_{\Tilde\Omega}^{\gamma_k}) \text{ for all $k>N$ and $\gamma_k\in\{\gamma'_k,\gamma''_k\}$.}
	\]
	The substitution into \eqref{eq-lup2} and \eqref{eq-eee2} results in
	\[
	E(R_{\Omega}^{\gamma_k})\le E(T^{\gamma_k})\le E(R_\Omega^{\gamma_k})+C \text{ for all $k>N$ and $\gamma_k\in\{\gamma'_k,\gamma''_k\}$,}
	\]
	and by combining with \eqref{eq-twoside2} we arrive at
	\begin{equation}
		\label{eq-twoside3}
		E(R_U^{\gamma_k})- C
		\le E(R_\Omega^{\gamma_k})\le E(R_U^{\gamma_k})+ C
		\text{ for all $k>N$, $\gamma_k\in\{\gamma'_k,\gamma''_k\}$.}
	\end{equation}
	
	Increase the value of $N$ again to obtain additionally
	\[
	\dfrac{C}{(\gamma'_k)^2}\le \delta \text{ and } \dfrac{C}{(\gamma''_k)^2}\le \delta \text{ for all }k> N.
	\]
	By combining \eqref{eq-twoside1} and \eqref{eq-twoside3} we obtain for all $k>N$:
	\begin{equation*}
		\begin{aligned}
			\dfrac{E(R_\Omega^{\gamma'_k})}{(\gamma'_k)^2}&\in \Big[-\dfrac{1}{\sin^2\theta'}-\dfrac{C}{(\gamma'_k)^2},-\dfrac{1}{\sin^2\theta'}+\delta+\dfrac{C}{(\gamma'_k)^2}\Big]\\
			&\quad\subset
			\Big[-\dfrac{1}{\sin^2\theta'}-\delta,-\dfrac{1}{\sin^2\theta'}+2\delta\Big],\\
			\dfrac{E(R_\Omega^{\gamma''_k})}{(\gamma''_k)^2}&\in \Big[-\dfrac{1}{\sin^2\theta''}-\dfrac{C}{(\gamma''_k)^2},-\dfrac{1}{\sin^2\theta''}+\delta+\dfrac{C}{(\gamma''_k)^2}\Big]\\
			&\quad\subset
			\Big[-\dfrac{1}{\sin^2\theta''}-\delta,-\dfrac{1}{\sin^2\theta''}+2\delta\Big].
		\end{aligned}
	\end{equation*}
	This shows the claim of Theorem \ref{thm1} for the above $\Omega$ with the intervals
	\begin{align*}
		I'&:=\Big[-\dfrac{1}{\sin^2\theta'}-\delta,-\dfrac{1}{\sin^2\theta'}+2\delta\Big],\\
		I''&:=\Big[-\dfrac{1}{\sin^2\theta''}-\delta,-\dfrac{1}{\sin^2\theta''}+2\delta\Big],
	\end{align*}
	and the sequences $\beta'_k:=\gamma'_{N+k}$ and $\beta''_k:=\gamma''_{N+k}$.

	\subsection{Concluding remarks}\label{ssec33}
	
		\begin{figure}[t]
		\centering

		\tikzset{every picture/.style={line width=0.75pt}} 
		
		\scalebox{0.7}{\begin{tikzpicture}[x=0.75pt,y=0.75pt,yscale=-1,xscale=1]
				
				\draw  [fill={rgb, 255:red, 155; green, 155; blue, 155 }  ,fill opacity=1 ][line width=1.5]  (15.67,119.75) -- (587,119.75) -- (587,214.75) -- (15.67,214.75) -- cycle ;
				\draw  [color={rgb, 255:red, 0; green, 0; blue, 0 }  ,draw opacity=1 ][fill={rgb, 255:red, 155; green, 155; blue, 155 }  ,fill opacity=1 ][line width=1.5]  (71.7,119.75) .. controls (74.15,120) and (80.5,120) .. (94.75,120) .. controls (109,120) and (105.25,69.9) .. (111.85,69.9) .. controls (118.45,69.9) and (111.45,119.95) .. (131.5,120.25) .. controls (151.55,120.55) and (151.25,120.75) .. (151.5,120.5) .. controls (151.75,120.25) and (153.25,221.75) .. (152.75,216.5) .. controls (152.25,211.25) and (71.45,212.25) .. (71.7,214.75) .. controls (71.95,217.25) and (69.25,119.5) .. (71.7,119.75) -- cycle ;
				\draw  [color={rgb, 255:red, 0; green, 0; blue, 0 }  ,draw opacity=1 ][fill={rgb, 255:red, 155; green, 155; blue, 155 }  ,fill opacity=1 ][line width=1.5]  (180.5,120.5) .. controls (180.8,121) and (180,120.25) .. (190.75,120.5) .. controls (201.5,120.75) and (213.4,89.5) .. (220,89.5) .. controls (226.6,89.5) and (243.5,120.5) .. (250.75,120.25) .. controls (258,120) and (260.5,120) .. (260.75,120.25) .. controls (261,120.5) and (263.3,222.75) .. (263,217.5) .. controls (262.7,212.25) and (182.8,218.25) .. (181.5,215.25) .. controls (180.2,212.25) and (180.2,120) .. (180.5,120.5) -- cycle ;
				\draw  [color={rgb, 255:red, 0; green, 0; blue, 0 }  ,draw opacity=1 ][fill={rgb, 255:red, 155; green, 155; blue, 155 }  ,fill opacity=1 ][line width=1.5]  (292.04,119.92) .. controls (292.19,120.18) and (296.46,120.18) .. (303.82,120.18) .. controls (311.19,120.18) and (309.38,94.15) .. (312.79,94.15) .. controls (316.2,94.15) and (312.59,120.02) .. (322.95,120.18) .. controls (333.32,120.33) and (333.16,120.44) .. (333.29,120.31) .. controls (333.42,120.18) and (332.75,203.12) .. (332.58,209.17) .. controls (332.41,215.21) and (293.86,211.7) .. (292.25,217.75) .. controls (290.64,223.8) and (291.88,119.66) .. (292.04,119.92) -- cycle ;
				\draw  [color={rgb, 255:red, 0; green, 0; blue, 0 }  ,draw opacity=1 ][fill={rgb, 255:red, 155; green, 155; blue, 155 }  ,fill opacity=1 ][line width=1.5]  (364.29,120.31) .. controls (364.45,120.57) and (364.03,120.18) .. (369.59,120.31) .. controls (375.15,120.44) and (381.3,104.28) .. (384.71,104.28) .. controls (388.12,104.28) and (396.86,120.31) .. (400.61,120.18) .. controls (404.36,120.05) and (405.65,120.05) .. (405.78,120.18) .. controls (405.91,120.31) and (406.25,192.83) .. (405.92,209.17) .. controls (405.58,225.5) and (367.25,208.83) .. (365.25,208.5) .. controls (363.25,208.17) and (364.14,120.05) .. (364.29,120.31) -- cycle ;
				\draw  [color={rgb, 255:red, 0; green, 0; blue, 0 }  ,draw opacity=1 ][fill={rgb, 255:red, 155; green, 155; blue, 155 }  ,fill opacity=1 ][line width=1.5]  (440.3,120.12) .. controls (440.39,120.27) and (442.92,120.27) .. (447.3,120.27) .. controls (451.67,120.27) and (450.6,104.82) .. (452.63,104.82) .. controls (454.65,104.82) and (452.5,120.18) .. (458.66,120.27) .. controls (464.81,120.37) and (464.72,120.43) .. (464.8,120.35) .. controls (464.88,120.27) and (464.92,188.83) .. (464.58,206.5) .. controls (464.25,224.17) and (439.58,212.5) .. (439.92,209.17) .. controls (440.25,205.83) and (440.21,119.97) .. (440.3,120.12) -- cycle ;
				\draw  [color={rgb, 255:red, 0; green, 0; blue, 0 }  ,draw opacity=1 ][fill={rgb, 255:red, 155; green, 155; blue, 155 }  ,fill opacity=1 ][line width=1.5]  (497.8,120.35) .. controls (497.89,120.5) and (497.64,120.27) .. (500.95,120.35) .. controls (504.25,120.43) and (507.9,110.83) .. (509.93,110.83) .. controls (511.95,110.83) and (517.14,120.35) .. (519.37,120.27) .. controls (521.59,120.2) and (522.36,120.2) .. (522.44,120.27) .. controls (522.51,120.35) and (522.04,201.46) .. (522.25,206.5) .. controls (522.46,211.54) and (496.92,206.5) .. (497.58,206.5) .. controls (498.25,206.5) and (497.71,120.2) .. (497.8,120.35) -- cycle ;
				\draw  [color={rgb, 255:red, 0; green, 0; blue, 0 }  ,draw opacity=1 ][fill={rgb, 255:red, 155; green, 155; blue, 155 }  ,fill opacity=1 ][line width=1.5]  (553.44,120.27) .. controls (553.48,120.34) and (554.59,120.34) .. (556.52,120.34) .. controls (558.45,120.34) and (557.97,113.53) .. (558.87,113.53) .. controls (559.76,113.53) and (558.81,120.3) .. (561.53,120.34) .. controls (564.24,120.38) and (564.2,120.41) .. (564.23,120.38) .. controls (564.26,120.34) and (564.26,193.15) .. (563.78,201.55) .. controls (563.29,209.94) and (555.11,206.21) .. (553.78,207.88) .. controls (552.44,209.55) and (553.4,120.21) .. (553.44,120.27) -- cycle ;
				\draw [color={rgb, 255:red, 128; green, 128; blue, 128 }  ,draw opacity=1 ][line width=3]    (71.85,121.25) -- (72.25,218.75) ;
				\draw [color={rgb, 255:red, 128; green, 128; blue, 128 }  ,draw opacity=1 ][line width=3]    (152.35,121.75) -- (152.5,217.25) ;
				\draw [color={rgb, 255:red, 128; green, 128; blue, 128 }  ,draw opacity=1 ][line width=3]    (261.5,121.75) -- (261.25,217.75) ;
				\draw [color={rgb, 255:red, 128; green, 128; blue, 128 }  ,draw opacity=1 ][line width=3]    (333.75,120.5) -- (333.5,210.5) ;
				\draw [color={rgb, 255:red, 128; green, 128; blue, 128 }  ,draw opacity=1 ][line width=3]    (406.5,120.5) -- (406.25,210.5) ;
				\draw [color={rgb, 255:red, 128; green, 128; blue, 128 }  ,draw opacity=1 ][line width=3]    (465.75,120.75) -- (465.5,210.75) ;
				\draw [color={rgb, 255:red, 128; green, 128; blue, 128 }  ,draw opacity=1 ][line width=2.25]    (522,121.5) -- (521.75,211.5) ;
				\draw [color={rgb, 255:red, 128; green, 128; blue, 128 }  ,draw opacity=1 ][line width=2.25]    (564.25,121.25) -- (564,211.25) ;
				\draw [color={rgb, 255:red, 255; green, 255; blue, 255 }  ,draw opacity=1 ][line width=3]    (15.67,116.67) -- (15.97,229.75) ;
				\draw [color={rgb, 255:red, 255; green, 255; blue, 255 }  ,draw opacity=1 ][line width=3]    (629.24,121.58) -- (629.25,211.25) ;
				\draw  [color={rgb, 255:red, 255; green, 255; blue, 255 }  ,draw opacity=1 ][fill={rgb, 255:red, 255; green, 255; blue, 255 }  ,fill opacity=1 ] (17,202.42) -- (631,202.42) -- (631,219) -- (17,219) -- cycle ;
				\draw [color={rgb, 255:red, 128; green, 128; blue, 128 }  ,draw opacity=1 ][line width=3]    (180.5,120.5) -- (181,202.67) ;
				\draw [color={rgb, 255:red, 128; green, 128; blue, 128 }  ,draw opacity=1 ][line width=3]    (292.04,120.92) -- (291.67,201.67) ;
				\draw [color={rgb, 255:red, 128; green, 128; blue, 128 }  ,draw opacity=1 ][line width=3]    (364.29,121.31) -- (364.33,203) ;
				\draw [color={rgb, 255:red, 128; green, 128; blue, 128 }  ,draw opacity=1 ][line width=3]    (440.3,121.12) -- (440.33,202.67) ;
				\draw [color={rgb, 255:red, 128; green, 128; blue, 128 }  ,draw opacity=1 ][line width=2.25]    (553.44,120.27) -- (553.33,202.33) ;
				\draw [color={rgb, 255:red, 128; green, 128; blue, 128 }  ,draw opacity=1 ][line width=2.25]    (498.44,121.27) -- (498,203) ;
				\draw [color={rgb, 255:red, 255; green, 255; blue, 255 }  ,draw opacity=1 ][line width=3]    (588.33,121) -- (588.33,203) ;
				\draw    (156,150.33) -- (177.33,150.33) ;
				\draw [shift={(180.33,150.33)}, rotate = 180] [fill={rgb, 255:red, 0; green, 0; blue, 0 }  ][line width=0.08]  [draw opacity=0] (8.93,-4.29) -- (0,0) -- (8.93,4.29) -- cycle    ;
				\draw [shift={(153,150.33)}, rotate = 0] [fill={rgb, 255:red, 0; green, 0; blue, 0 }  ][line width=0.08]  [draw opacity=0] (8.93,-4.29) -- (0,0) -- (8.93,4.29) -- cycle    ;
				\draw    (266.33,150.33) -- (287.67,150.33) ;
				\draw [shift={(290.67,150.33)}, rotate = 180] [fill={rgb, 255:red, 0; green, 0; blue, 0 }  ][line width=0.08]  [draw opacity=0] (8.93,-4.29) -- (0,0) -- (8.93,4.29) -- cycle    ;
				\draw [shift={(263.33,150.33)}, rotate = 0] [fill={rgb, 255:red, 0; green, 0; blue, 0 }  ][line width=0.08]  [draw opacity=0] (8.93,-4.29) -- (0,0) -- (8.93,4.29) -- cycle    ;
				\draw    (338,149.33) -- (360,149.33) ;
				\draw [shift={(363,149.33)}, rotate = 180] [fill={rgb, 255:red, 0; green, 0; blue, 0 }  ][line width=0.08]  [draw opacity=0] (8.93,-4.29) -- (0,0) -- (8.93,4.29) -- cycle    ;
				\draw [shift={(335,149.33)}, rotate = 0] [fill={rgb, 255:red, 0; green, 0; blue, 0 }  ][line width=0.08]  [draw opacity=0] (8.93,-4.29) -- (0,0) -- (8.93,4.29) -- cycle    ;
				\draw    (412.33,149.7) -- (435,149.97) ;
				\draw [shift={(438,150)}, rotate = 180.67] [fill={rgb, 255:red, 0; green, 0; blue, 0 }  ][line width=0.08]  [draw opacity=0] (8.93,-4.29) -- (0,0) -- (8.93,4.29) -- cycle    ;
				\draw [shift={(409.33,149.67)}, rotate = 0.67] [fill={rgb, 255:red, 0; green, 0; blue, 0 }  ][line width=0.08]  [draw opacity=0] (8.93,-4.29) -- (0,0) -- (8.93,4.29) -- cycle    ;
				
				\draw (37,149.73) node [anchor=north west][inner sep=0.75pt]    {$U_{0}$};
				\draw (156.33,180.73) node [anchor=north west][inner sep=0.75pt]    {$P_{1}$};
				\draw (162,131.4) node [anchor=north west][inner sep=0.75pt]    {$1$};
				\draw (271.67,131.07) node [anchor=north west][inner sep=0.75pt]    {$1$};
				\draw (343.67,130.73) node [anchor=north west][inner sep=0.75pt]    {$1$};
				\draw (418,130.73) node [anchor=north west][inner sep=0.75pt]    {$1$};
				\draw (267.33,179.73) node [anchor=north west][inner sep=0.75pt]    {$P_{2}$};
				\draw (342.33,180.4) node [anchor=north west][inner sep=0.75pt]    {$P_{3}$};
				\draw (416.33,180.4) node [anchor=north west][inner sep=0.75pt]    {$P_{4}$};
				\draw (100,151.73) node [anchor=north west][inner sep=0.75pt]    {$U_{1}$};
				\draw (209,151.73) node [anchor=north west][inner sep=0.75pt]    {$U_{2}$};
				\draw (302,151.73) node [anchor=north west][inner sep=0.75pt]    {$U_{3}$};
				\draw (374,151.73) node [anchor=north west][inner sep=0.75pt]    {$U_{4}$};

			\end{tikzpicture}}

		\caption{A non-compact but smooth ``bad'' domain $V$.}\label{fig:vv}
	\end{figure}
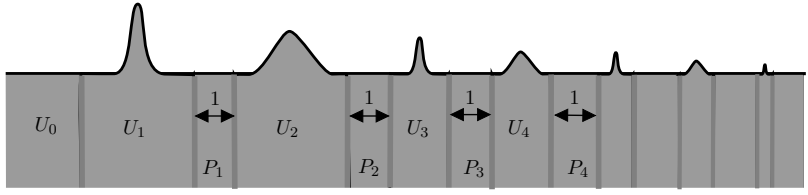
	
				By slightly modifying the arguments in Subsection \ref{ssec31} one can also construct a non-compact but $C^\infty$ smooth domain $V$ such that $E(R_V^\alpha)$ is well-defined for all $\alpha\in\RR$ but the ratio $E(R_V^\alpha)/\alpha^2$ has no limit for $\alpha\to-\infty$. Namely, let us take the same domains $U_n$ as above and put them next to each other in the order $U_0,U_1,P_1,U_2,P_2,U_3,P_3,\dots$, where each $P_n$ is an isometric copy of $(0,1)\times (-\infty,0)$, see Fig.~\ref{fig:vv}. One uses again the Dirichlet-Neumann bracketing along the gluing lines as in Subsection \ref{ssec31} and notes that the decomposition produces Laplacians $P_n$ admitting a separation of variables with spectra in $[-\alpha^2,\infty)$. Then the two-sided estimate \eqref{eq-twoside1} remains valid if one replaces $E(R_U^\alpha)$ with $E(R_V^\alpha)$,
	and the same analysis implies that for any $k\in\NN$ one has the inclusions
	\[
	\begin{aligned}
		\dfrac{E(R_V^{\gamma'_k})}{(\gamma'_k)^2}&\in \Big[-\dfrac{1}{\sin^2\theta'},-\dfrac{1}{\sin^2\theta'}+\delta\Big],\\
		\qquad \dfrac{E(R_V^{\gamma''_k})}{(\gamma''_k)^2}&\in \Big[-\dfrac{1}{\sin^2\theta''},-\dfrac{1}{\sin^2\theta''}+\delta\Big],
	\end{aligned}
	\]
	which shows that $V$ possesses the required properties. As the curvature of $\partial V$ is unbounded,  the non-existence
	of $\lim\limits_{\alpha\to-\infty}E(R_V^\alpha)/\alpha^2$  does not contradict the existing results on the eigenvalues of Robin Laplacians in smooth unbounded domains~\cite{EM,kop1}.
	
	\bigskip
	
\section*{Acknowledgments}
The work started during the conference ``Asymptotic analysis and spectral theory'' (Aspect 2024) that took place in Metz in September 2024. The authors thank the conference organizers J\'er\'emy Faupin, Victor Nistor and Olaf Post for creating the stimulating atmosphere during the meeting.

C.D. expresses her deepest gratitude to Phan Thành Nam and
Laure Saint-Raymond for their continued support. This work was done while C.D. stayed at the Institut des Hautes \'Etudes Scientifiques, and she would like to thank the institute team for the warm hospitality. She acknowledges the support by the European Research Council via ERC CoG RAMBAS, Project No.\ 101044249.

The figures were prepared with the editor \url{https://mathcha.io}.

\end{document}